\documentclass[12pt]{amsart}
\usepackage{amsfonts}
\usepackage{amssymb}

\newtheorem{prop}{\sc Proposition}
\newtheorem{lem}{\sc Lemma}
\newtheorem{thm}{\sc Theorem}
\newtheorem{cor}{\sc Corollary}
\newtheorem{ex}{\sc Example}

\newtheorem{other}{\sc Theorem}              % Other papers' theorems
\newenvironment{pf}{\noindent{\textit{Proof. }}}{$\Box$ }

\def\C{{\mathbb C}}

\def\D{{\mathbb D}}

\begin{document}

\title [Pre-Schwarzian and Schwarzian derivatives of harmonic mappings]{Pre-Schwarzian and Schwarzian derivatives of harmonic mappings}

\author[R. Hern\'andez]{Rodrigo Hern\'andez}
\address{Facultad de Ingenier\'{\i}a y Ciencias, Universidad Adolfo Ib\'a\~nez, Avda. Padre Hurtado 750,
Vi\~na del Mar, Chile.} \email{rodrigo.hernandez@uai.cl}

\author[M. J. Mart\'{\i}n]{Mar\'{\i}a J. Mart\'{\i}n}
\address{Departamento de Matem\'aticas (M\'odulo 17, Edificio de Ciencias),
Universidad Aut\'onoma de Ma\-drid, 28049 Madrid, Spain.}
\email{mariaj.martin@uam.es}
\urladdr{http://www.uam.es/mariaj.martin}

\thanks{The authors were partially supported by Fondecyt Grant \# 1110160 (Chile) and Grant MTM2009-14694-C02-01,
MICINN (Spain).
\endgraf  {\sl Key words:} Pre-Schwarzian derivative, Schwarzian derivative, harmonic mappings, univalence, Becker's criterion,
convexity.
\endgraf {\sl 2010 AMS Subject Classification}. Primary: 30C55;\,
Secondary: 31A05, 30C45.}

% ----------------------------------------------------------------
\begin{abstract}

In this paper we introduce a definition of the pre-Schwarzian
and the Schwarzian derivatives of \emph{any} locally univalent harmonic
mapping $f$ in the complex plane without assuming any additional
condition on the (second complex) dilatation $\omega_f$ of $f$. Using the new definition for the Schwarzian derivative of harmonic mappings, we prove analogous theorems to those by Chuaqui, Duren, and Osgood. Also, we obtain a Becker-type criterion for the
univalence of harmonic mappings.

\end{abstract}
\maketitle
\section{Introduction}
\label{intro}
\par
Let $f$ be a locally univalent analytic function defined on a simply
connected domain $\Omega\subset \C$. The \emph{pre-Schwarzian
derivative} of $f$ is defined by
\begin{equation}\label{eq-def-preSch-analytic}
Pf=\frac{f''}{f'}\,.
\end{equation}
\par
The \emph{Schwarzian derivative} of such a function $f$ equals
\begin{equation}\label{eq-def-Sch-analytic} Sf=\left(Pf\right)'-\frac
12\left(Pf\right)^2\,.
\end{equation}
\par
It is well known that $P(A\circ f)=Pf$ for all linear (or affine)
transformations $A(w)=aw+b$, $a\neq 0$. Also, that the Schwarzian
derivative is invariant under (non-constant) linear fractional (or
M\"{o}bius) transformations. In other words, $S(T\circ f)=S(f)$ for
every  function $T$ of the form
$$T(w)=\displaystyle\frac{aw+b}{cw+d},\quad ad-bc\neq 0\,.$$
\par\smallskip
This are just particular cases of the \emph{chain rule} for the pre-Schwarzian and Schwarz\-ian derivatives:
\begin{equation}\label{eq-chainrule}
P(g\circ f)= (Pg\circ f)f^\prime+Pf\quad and \quad S(g\circ f)=(Sg\circ f)(f')^2+Sf\,,
\end{equation}
respectively, which hold whenever the composition $g\circ f$ is defined.
\par
The pre-Schwarzian and the Schwarzian derivatives of locally
univalent analytic mappings $f$ are widely used tools in the study
of geometric properties of such functions. For instance, it can be
used to get either necessary or sufficient conditions for the global
univalence, or to obtain certain geometric conditions on the range
of $f$. More specifically, it is well known that any
\emph{univalent} analytic transformation $f$ in the unit disk $\D$
satisfies the sharp inequality
\[
|Pf(z)|\leq \frac{6}{1-|z|^2}\,, \quad z\in\D\,.
\]
Also, the \emph{Becker's univalence criterion} \cite{B} states that
if $f$ is locally univalent (and analytic) in $\D$ and
\[
\sup_{z\in\D} |zPf(z)|\, (1-|z|^2)\leq 1\,,
\]
then $f$ is univalent in the unit disk. Becker and Pommerenke \cite{BP} proved later that the constant $1$ is sharp.
\par
In terms of the Schwarzian derivative, we have the sharp inequality
\[
|Sf(z)|\leq \frac{6}{(1-|z|^2)^2} \quad (|z|<1)
\]
that holds for all univalent analytic functions $f$ in $\D$. This
result is due to Krauss \cite{Krauss} and was rediscovered by Nehari
\cite{Nehari} who also proved that if a locally univalent analytic
function $f$ in the unit disk satisfies
\begin{equation}\label{eq-Nehari}
|Sf(z)|\leq \frac{2}{(1-|z|^2)^2}\,, \quad |z|<1\,,
\end{equation}
then $f$ is univalent. The constant $C=2$ is sharp. Moreover, whenever the Schwarzian derivative of a function $f$ as above satisfies $|Sf(z)| (1-|z|^2)^2 \leq C$ in the unit disk for some constant $C<2$, then not only is $f$ univalent in $\D$ but it also maps the disk onto a Jordan domain on the Riemann sphere (see \cite{AW} and \cite{DL}).
\par\smallskip
It is possible to prove (\ref{eq-Nehari}) using the connection of the Schwarzian derivative with certain linear differential equation which indeed solves the problem of finding the most general analytic function with prescribed Schwarzian derivative: A function $f$ has Schwarzian derivative $Sf=2p$ if and only if it has the form $f=u_1/u_2$ for some pair of independent solutions $u_1$ and $u_2$ of the linear differential equation $u^{\prime\prime}+pu=0$. As a corollary, it is easy to obtain that if $Sf=Sg$ then $g=T\circ f$ for some linear fractional transformation $T$. In particular, the only analytic functions with $Sf=0$ are the M\"{o}bius transformations. There is still another consequence: any analytic function is the Schwarzian derivative of some locally univalent function $f$.
\par
For more properties and results related to the pre-Schwarzian and
Schwarzian derivatives of locally univalent analytic mappings, we
refer the reader to the monographs \cite{P}, \cite{Dur-Univ}, or
\cite{GK}.
\par\medskip
A definition of the Schwarzian derivative for a more general class of complex-valued \emph{harmonic} functions was presented by Chuaqui, Duren, and Osgood in \cite{CHDO-DefSchw}. The formula was derived by passing to the minimal surface associated locally with a given harmonic function and appealing to a definition given in \cite{OS} for the Schwarzian derivative of a conformal mapping between arbitrary Riemannian manifolds. Specifically, recall that a complex-valued harmonic function in a simply connected domain has a canonical representation $f=h+\overline g$ (where $h$ and $g$ are analytic functions) that is unique up to an additive constant. Any harmonic mapping  $f$ with $|h^\prime|+|g^\prime|\neq 0$ lifts to a mapping $\widetilde{f}$ onto a minimal surface defined by conformal parameters if and only if the dilatation $\omega=g^\prime/h^\prime$ equals the square of an analytic function $q$. In other words, $\omega=q^2$ for some analytic function $q$. For such mappings $f$, the Schwarzian derivative presented in \cite{CHDO-DefSchw} (that we denote by $\mathbb{S}$) is defined by the formula
\[
\mathbb{S}f=2(\sigma_{zz}-\sigma_z^2)\,,
\]
where $\sigma=\log(|h^\prime|+|g^\prime|)$ and $$\sigma_z=\frac{\partial \sigma}{\partial z}=\frac 12 \left(\frac{\partial \sigma}{\partial x}-i\frac{\partial \sigma}{\partial y}\right)\,,\quad z=x+iy\,.$$
\par
In terms of the canonical representation $f=h+\overline g$ and the dilatation $\omega=q^2$,
\begin{equation}\label{eq-oldSch}
\mathbb{S}f=Sh+\frac{2\overline q}{1+|q|^2}\left(q^{\prime\prime}-q^\prime \frac{h^{\prime\prime}}{h^\prime}\right)-4\left(\frac{q^\prime \overline q}{1+|q|^2}\right)^2\,,
\end{equation}
where $Sh$ is the classical Schwarzian derivative of the analytic function $h$. The Sch\-warz\-ian derivative $\mathbb{S}f$ is well-defined in a neighborhood of any point where $\omega$ has a zero of odd order. If $q$ has a zero of order at least $2$ at a point $z_0$, then $\mathbb{S} f$ tends to $Sh$ as $z\to z_0$ (see \cite{CHDO-DefSchw}). The Schwarzian derivative $\mathbb{S} f$ is well-defined for some harmonic mappings $f$ that are not locally univalent. For instance, as an extreme example of this fact, we can consider the harmonic mapping $f(z)=z+\overline z=2 Re\{z\}, z\in\D$. Since $|h^\prime|+|g^\prime|\neq 0$, we can define $\mathbb{S}f$. Indeed, it is easy to see that $\mathbb{S}f (z)\equiv 0$ in $\D$.
\par
Chuaqui, Duren, and Osgood, using their definition (\ref{eq-oldSch}) for the Schwarzian derivative of harmonic mappings, have obtained many interesting results related to different questions as some curvature properties, the univalence of the Weierstrass-Enneper lift, or a characterization of the uniform local univalence (with respect to the hyperbolic metric), among others (see \cite{Ch-D-O-04,Ch-D-O-05,Ch-D-O-07,Ch-D-O-07Phill,Ch-D-O-08}).
\par\smallskip
The Schwarzian derivative (\ref{eq-oldSch}) has one disadvantage
that arises from the fact that the dilatation is required to fulfill
an extra condition. This causes that in many cases one cannot define
(globally) the Schwarzian derivative of a univalent harmonic mapping
$f$ in a simply connected domain $\Omega$. For instance, although it
is known \cite{Ch-D-O-07Phill} that any univalent harmonic mapping
$f$ in the unit disk $\D$ with dilatation $\omega_f=q^2$ satisfies
\[
|\mathbb{S}f(z)|\leq \frac{C}{(1-|z|^2)^2}\,, \quad |z|<1\,,
\]
where $C$ is a constant with unknown sharp value, it is not possible
to state an analogous result for the hole family of harmonic
mappings in the unit disk in terms of $\mathbb{S}$.
\par
It should perhaps be pointed out that also, even in the case when
$f$ is a univalent harmonic mapping in $\Omega$ with dilatation
$\omega_f=q^2$, the Schwarzian derivative of the composition
$F=A\circ f$, where $A$ is an \emph{affine harmonic mapping} given
by $A(w)=a w+b\overline w+c$ ($a, b, c \in\C$), may fail to be
defined at some point $z_0\in\Omega$ because of the fact that the
dilatation $\omega_F$ of $F$ can have a simple zero at $z_0$.
Keeping in mind that the correspondence between the well-known
families of univalent harmonic mappings suitably normalized, $S_H$
(that is a normal family, but non-compact), and $S_H^0$ (normal and
compact) is realized via affine harmonic mappings (see
\cite[Sections 5.1 and 5.2]{Dur-Harm}), the fact that $\mathbb{S} F$
is not defined at $z_0$ results in the loss of normality argument
tools to solve certain extremal problems involving the Schwarzian
derivative (and related to univalent harmonic mappings).
\par\medskip
The main purpose of this paper is to present another definition of
the Schwarzian derivative (that will be denoted by $S_f$ to
differentiate it from the definitions of the Schwarzian derivatives
already mentioned) for \emph{any} locally univalent harmonic mapping
$f$ in a simply connected domain in the complex plane. We will
justify the new definition in two different analytic ways: the first
one consists in applying the interpretation of the classical
Schwarzian derivative (\ref{eq-def-Sch-analytic}) given by Tamanoi
in \cite{T} to harmonic M\"obius transformations. The second one is
based on a relation between $S_f$ and the classical formula
(\ref{eq-def-Sch-analytic}) for locally univalent analytic functions
 (in terms of the Jacobian). This relation will be used to define the \emph{pre-Schwarzian derivative}
 $P_f$ of a locally univalent harmonic mapping.
\par
From the geometrical point of view, we will see that $S_f$ equals
the Schwarzian tensor of certain conformal metric relative to the
Euclidean metric (just like $\mathbb{S}$).
\par\smallskip
We will prove that not only $S_f$ satisfies all the properties that
$\mathbb{S} f$ does (characterization of the harmonic M\"{o}bius
transformations, chain rule, \ldots), but also others that seem to
be natural in the context of the theory of harmonic mappings. For
instance, that $S_{(A\circ f)}=S_f$ for any affine harmonic mapping
$A$ as above and any locally univalent harmonic mapping $f$ in a
simply connected domain. Moreover, although $\mathbb{S} f$ and $S_f$
are different (in the cases when $\mathbb{S} f$ is defined), the
quantities
\[
\|\mathbb{S} f\|=\sup_{|z|<1}|(\mathbb{S} f)(z)| (1-|z|^2)^2\quad \emph{and} \quad \|S_f\|=\sup_{|z|<1}|S_f(z)| (1-|z|^2)^2
\]
are simultaneously finite or infinite. Hence, we are able to prove the theorems by Chuaqui, Duren, and Osgood involving $\|\mathbb{S} f\|$ and related to uniform local univalence in terms of the norm $\|S_f\|$ of the ``new'' Schwarzian derivative (in many cases, using the same arguments as them). The statements of the corresponding theorems in this paper hold for the full class of locally univalent harmonic mappings in the unit disk with no extra assumption on the dilatation.
\par
We will also prove other additional results. For instance, we will
show a criterion of univalence for sense-preserving harmonic
mappings that generalizes the classical Becker's criterion.
\par\medskip
In terms of the canonical representation of a sense-preserving harmonic mapping $f=h+\overline g$ in a simply connected domain $\Omega$ with dilatation $\omega=g^\prime/h^\prime$, we define $S_f$ (in $\Omega)$ by
\[
S_f=Sh+\frac{\overline \omega}{1-|\omega|^2}\left(\frac{h''}{h'}\,\omega'-\omega''\right)
-\frac 32\left(\frac{\omega'\,\overline
\omega}{1-|\omega|^2}\right)^2\,.
\]

%xxxxxxxxxxxxxxxxxxxxxxxxxxxxxxxxxxxxxxxxxxxxxxxxxxxxxxxxxxxxxxxxxxxxxxxxxxxxxxxxxxxxxxxxxxxxxxxxxxxxxxxxxxxxxxxxxxxxxxxxxxxxxx

\section{Background}\label{sec-background}

% xxxxxxxxxxxxxxxxxxxxxxxxxxxxxxxxxxxxxxxxxxxxxxxxxxxxxxxxxxxxxxxxxxxxxxxxxxxxxxxxxxxxxxxxxxxxxxxxxxxxxxxxxxxxxxxxxxxxxxxxxxxxxxxx

%%%%%%%%%%%%%%%%%%%%%%%%%%%%%%%%%%%%%%%%%%%%%%%%%%%%%%%%%%%%%%%%%%%%%%%%%%%%%%%%%%%
\subsection{Harmonic mappings in the complex plane}\label{ssec-harmonicmaps}
%%%%%%%%%%%%%%%%%%%%%%%%%%%%%%%%%%%%%%%%%%%%%%%%%%%%%%%%%%%%%%%%%%%%%%%%%%%%%%%%%%%
A straightforward calculation in terms of the well-known Wirtinger operators
$$\frac{\partial}{\partial z}=\frac 12\left(\frac{\partial }{\partial x}-i\frac{\partial}{\partial y}\right)\,,\quad\quad \frac{\partial}{\partial \overline z}=\frac 12\left(\frac{\partial }{\partial x}+i\frac{\partial}{\partial y}\right)\,,\quad z=x+iy\,,$$
shows that the Laplacian of $f$ is
\begin{equation}\label{eq-laplacian}
\triangle f=4\frac{\partial^2 f}{\partial z\partial \overline z}\,.
\end{equation}
Therefore, for functions $f$ with continuous second partial derivatives, we have that $f$ is harmonic if and only if $\partial f/\partial z$ is analytic. This fact proves that a complex-valued harmonic function $f$ in a simply connected domain $\Omega\in\C$ has the representation $f=h+\overline
g$ (where $h$ and $g$ are analytic functions in $\Omega$) that is unique up to an additive constant. For a harmonic mapping $f$ of the unit disk $\D$, it is convenient to choose the additive constant so that $g(0)=0$. The representation $f=h+\overline
g$ is therefore unique and is called the \emph{canonical representation} of $f$.
\par\smallskip
It is a consequence of the inverse mapping theorem that if the Jacobian of a $C^1$ mapping $f:\mathbb{C}\to \mathbb{C}$ does not vanish, then $f$ is locally univalent. A theorem of Lewy \cite{Lewy} asserts that the converse is also true for harmonic mappings. Specifically, a harmonic mapping $f=h+\overline g$ in a simply connected domain $\Omega$ is locally univalent if and only if its Jacobian, $J_f=|h^\prime|^2-|g^\prime|^2$, does not vanish in $\Omega$. By Lewy's theorem, harmonic mappings are either \emph{sense-preserving} or \emph{sense-reversing} depending on the conditions  $J_f>0$  and  $J_f<0$ throughout the domain $\Omega$ where $f$ is locally univalent, respectively. If $f$ is sense-preserving, then $\overline f$ is sense-reversing. Locally univalent analytic functions are sense-preserving. Notice that if $f=h+\overline g$ is sense-preserving, its analytic part $h$ is locally univalent: $h^\prime\neq 0$ in $\Omega$, and the second complex dilatation of $f$, $\omega_f=g^\prime/h^\prime$, is an analytic function in $\Omega$ with $|\omega_f|<1$.
%%%%%%%%%%%%%%%%%%%%%%%%%%%%%%%%%%%%%%%%%%%%%%%%%%%%%%%%%%%%%%%%%%%%%%%%%%%%%%%%%%%%%%
\subsection{Univalent harmonic mappings in $\D$. The shear construction}\label{ssec-shear}
%%%%%%%%%%%%%%%%%%%%%%%%%%%%%%%%%%%%%%%%%%%%%%%%%%%%%%%%%%%%%%%%%%%%%%%%%%%%%%%%%%%%%%
A harmonic mapping $f$ in $\Omega$ is \emph{univalent} if
$f(z_1)\neq f(z_2)$ for all $z_1,\ z_2\in\Omega$ with $z_1\neq z_2$.
Let $f$ be a univalent harmonic mapping in the unit disk and
$f=h+\overline g$ its canonical representation (with $g(0)=0$). If
$f$ is sense-preserving, then $h$ is locally univalent, so there is
no loss of generality in assuming that $h(0)=1-h^\prime(0)=0$. The
class of all sense-preserving univalent harmonic mappings in the
unit disk with the normalizations $g(0)=h(0)=1-h^\prime(0)=0$ is
denoted by $S_H$. This class $S_H$ is not compact (see \cite[p.
78]{Dur-Harm}). To produce a compact normal family it is enough to
use one further normalization: the family $S_H^0=\{f\in S_H \colon
g^\prime(0)=0\}$ is compact. Note that if $f=h+\overline g\in S_H$
and $g^\prime(0)=b_1\in\D$, the composition $f_0=\tau\circ f$, where
\[
\tau(w)=\frac{w-\overline{b_1} \overline w}{1-|b_1|^2},
\]
belongs to $S_H^0$. Conversely, given any value of $b_1\in\D$, there
is a unique function $f=f_0+\overline{b_1 f_0}\in S_H$ with
$g^\prime(0)=b_1$ that corresponds to a given function $f_0\in
S_H^0$. In other words, the correspondence between functions in the
families $S_H$ and $S_H^0$ is realized using affine harmonic
mappings of the form $A(w)=a w+b\overline w$ with $a\neq 0$ and $|b|<|a|$.
\par\smallskip
An effective way of constructing sense-preserving univalent harmonic
mappings in $\D$ is the so called \emph{shear construction}
introduced in the paper \cite{CSS} by Clunie and Sheil-Small. Recall
that a function $f$ is convex in the $\theta$ direction ($0\leq
\theta<2\pi$) if the intersection of the domain $f(\D)$ with every
line parallel to the line through $0$ and $e^{i\theta}$ is either
the empty set or an interval. In \cite{CSS}, it was proved that
given a locally univalent harmonic mapping $f=h+\overline g$ in the
unit disk, then $f$ is univalent and convex in the $\theta$
direction if and only if the analytic function $h-e^{2i\theta} g$ is
univalent and convex in the $\theta$ direction. This result was used
in \cite{CSS} to construct two explicit examples of univalent
harmonic mappings in $\D$: the \emph{harmonic Koebe function} and
the \emph{half-plane harmonic mapping}. Many other examples of
univalent harmonic mappings constructed using the shear construction
are shown in \cite[Section 3.4]{Dur-Harm}. We review some of them.
\par\smallskip
The \emph{harmonic Koebe function} is defined by $K=h+\overline g$ where
\begin{eqnarray}\label{eq-Koebe-harmonic}
h(z)=\frac{z-\frac 12 z^2+\frac 16 z^3}{(1-z)^3},\quad g(z)=\frac{\frac 12 z^2+\frac 16 z^3}{(1-z)^3} \quad (z\in\D)\,.
\end{eqnarray}
This function $K$ maps $\D$ onto the full plane minus the part of the negative real axis from $-1/6$ to infinity and can be obtained as the horizontal shear (that is, the shear in the $0$ direction) of the Koebe mapping $k(z)=z/(1-z)^2$ with dilatation $\omega(z)=z$. For many reasons, it is thought that the harmonic Koebe mapping $K$ is the probable analogue of the Koebe function $k$ for the class $S_H^0$ of harmonic univalent functions.
\par\smallskip
The \emph{half-plane harmonic mapping} $L=h+\overline g$ was defined in \cite{CSS} as the vertical shear (the shear in the $\pi/2$ direction) of the function $\ell (z)=z/(1-z)$ with dilatation $\omega(z)=-z$. The function $L$ maps the disk onto the half-plane $\{z\in\C\colon Re\{z\}>-1/2\}$. The explicit formulas for $h$ and $g$ are
\begin{eqnarray}\label{eq-formulaL}
h(z)=\frac{2z-z^2}{2(1-z)^2},\quad g(z)=\frac{-z^2}{2(1-z)^2}\,.
\end{eqnarray}
\par\smallskip
Two other univalent harmonic mappings will be important for our purposes: The \emph{half-strip harmonic mapping} $S_1$ and the \emph{strip harmonic mapping} $S_2$. Both functions are horizontal shears of the conformal mapping
\begin{equation}\label{eq-strip-analytic}
s(z)=\frac 12 \log\frac{1+z}{1-z}\,
\end{equation}
that maps $\D$ onto the domain $\Omega=\{z\in\C\colon |Im\{z\}|<\pi/4\}$.
The mapping $S_1$ has dilatation $w(z)=z$. The dilatation of $S_2$ is $w(z)=z^2$. Concretely, it can be seen that
\begin{eqnarray}\label{eq-formulaS1}
S_1=h_1+\overline{g_1}, \quad \emph{where}\quad h_1=\frac 12 (\ell+s) \quad \emph{and }\quad g_1(z)=\frac 12 (\ell-s)\,
\end{eqnarray}
in the unit disk. Moreover,
\[
S_1(\D)=\left\{w\in\C\colon Re\{w\}>-\frac 12,\quad |Im\{w\}|<\frac \pi 4\right\}\,.
\]
The harmonic mapping $S_2=h_2+\overline g_2$ is defined by
\begin{eqnarray}\label{eq-formulaS2}
\quad h_2=\frac 12 (q+s) \ \  \emph{and}\ \  g_2(z)=\frac 12 (q-s), \
\end{eqnarray}
where $q(z)=\sqrt{k(z^2)}=z/(1-z^2)$,  $|z|<1$. Note that $S_2(\D)=s(\D)$.
\par
Each function $L$, $S_1$, and $S_2$ maps $\D$ onto a convex domain.

% xxxxxxxxxxxxxxxxxxxxxxxxxxxxxxxxxxxxxxxxxxxxxxxxxxxxxxxxxxxxxxxxxxxxxxxxxxxxxxxxxxxxxxxxxxxxxxxxxxxxxxxxxxxxxxxxxxxxxxxxxxxxxxxx

\section{Two Analytic derivations of the formula for the Schwarzian derivative of locally univalent harmonic mappings}\label{sec-definitionSchwarzian}

% xxxxxxxxxxxxxxxxxxxxxxxxxxxxxxxxxxxxxxxxxxxxxxxxxxxxxxxxxxxxxxxxxxxxxxxxxxxxxxxxxxxxxxxxxxxxxxxxxxxxxxxxxxxxxxxxxxxxxxxxxxxxxxxx
\par
In this section, we derive the formula for the Schwarzian derivative $S_f$ mentioned in the introduction in two different analytic ways. We also see in Subsection~\ref{ssec-tensor} that $S_f$ equals the Schwarzian tensor of certain conformal metric (relative to the Euclidean metric).
%%%%%%%%%%%%%%%%%%%%%%%%%%%%%%%%%%%%%%%%%%%%%%%%%%%%%%%%%%%%%%%%%%%%%%%%%%%%%%%%%%%%%%%%%%%%%%%%%%%%%%
\subsection{Tamanoi's interpretation of the classical Schwarzian derivative}\label{ssec-Tamanoi}
%%%%%%%%%%%%%%%%%%%%%%%%%%%%%%%%%%%%%%%%%%%%%%%%%%%%%%%%%%%%%%%%%%%%%%%%%%%%%%%%%%%%%%%%%%%%%%%%%%%%%%
In \cite{T}, Tamanoi studies M\"{o}bius invariant differential operators which arise in
the context of approximating locally univalent analytic functions  at the origin by M\"{o}bius transformations. In particular, the classical Schwarz\-ian derivative. In this section, we review some of the arguments developed in \cite{T} and use them to propose our new definition of the Schwarzian derivative for locally univalent harmonic mappings.
\par\smallskip
For a locally univalent analytic function $f$ in the simply connected domain $\Omega$ with $0\in\Omega$, Tamanoi defined the \emph{best M\"{o}bius approximation of $f$ at the origin} as the linear fractional transformation
\[
T_f(z) =\frac{az + b}{cz + d},\quad ad-bc\neq 0\,,
\]
which approximates $f(z)$ at $z=0$ up to the second order: $T_f(0)=f(0)$, $T^\prime_f(0) =f^\prime(0)$, and $T^{\prime\prime}_f(0)=f^{\prime\prime}(0)$. Since the family of M\"{o}bius transformations is the complex $3$-dimensional group $PSL_2(\C)$, the three conditions given above determine $T_f$ uniquely.
\par The \emph{deviation of a locally univalent analytic function $f$ from their best M\"{o}bius transformation} is the function $F_f=T_f^{-1}\circ f$. Let the Taylor expansion of $F_f$ at the origin be
\[
F_f(z)=\sum_{n=0}^\infty
s_n\frac{z^{n}}{n!}=z+\frac{1}{3!} Sf(0)\, z^3+\ldots\,,
\]
where $s_n=s_n(f)$ are the Taylor coefficients of $F_f$ and $S$ is the classical Schwarzian derivative. The Taylor coefficients $s_n$ are invariant under pre-composition with M\"{o}bius transformations (see \cite[Lemma~3-1]{T}). In other words, $s_n(T\circ f)=s_n(f)$ for all (non-constant) M\"{o}bius transformations $T$ and each non-negative integer $n$. Note that $SF_f(0)=6\left(s_3-{s_2}^2\right)=Sf(0)$ since, in this case, $s_2=0$.
\par
For a point $w\in\Omega\setminus\{0\}$, consider the analytic function $f_w=f(w+z)$ which is well defined in the domain $\Omega_w=\{z\in\C\colon z-w\in\Omega\}$ (that contains the origin). Let $T_{f_w}$ be the best M\"{o}bius approximation of $f_w$ at $z=0$. It can be checked (see \cite[p.~136]{T}) that
\begin{eqnarray*}
F_{f_w}(z)=T_{f_w}^{-1}\circ f_w(z)=\sum_{n=0}^\infty s_n(f)(w)\frac{z^{n}}{(n)!}= z+\frac{1}{3!} Sf(w)\, z^3+\ldots
\end{eqnarray*}
Therefore,
\[
Sf(w)=SF_f(w)=6\left(s_3(f)(w)-(s_2(f)(w))^2\right)\,.
\]
\par
We will follow the same argument as before to derive a formula for the Schwarzian derivative of a sense-preserving harmonic mapping in the complex plane.
\par\medskip
A \emph{sense-preserving harmonic M\"{o}bius transformation $M$} is a transformation of the form
\[
M=T+\alpha\overline T\,,
\]
where $T$ is a (classical) non-constant M\"{o}bius transformation and $\alpha$ is a constant with $|\alpha|<1$. Let $f=h+\overline g$ be a sense-preserving harmonic mapping in a simply connected domain $\Omega\subset \C$ that contains the origin. We define \emph{best harmonic M\"{o}bius approximation of $f$ at the origin} as the harmonic M\"{o}bius transformation $M_f(z)$ that satisfies $M_f(0)=f(0),$
\[
\quad \frac{\partial M_f}{\partial z}(0)=\frac{\partial f}{\partial z}(0)=h'(0),\quad \frac{\partial M_f}{\partial\overline z}(0)=\frac{\partial f}{\partial\overline z}(0)=\overline{g'(0)}\,,
\]
and
\[
\frac{\partial^2 M_f}{\partial z^2}(0)=\frac{\partial^2 f}{\partial z^2}(0)=h''(0)\,.
\]
Notice that these four conditions determine $M_f$ uniquely. Moreover, recalling the connection between the derivatives $\partial/\partial z$ and $\partial/\partial \overline z$ given by
\[
\overline{\left(\frac{\partial f}{\partial z}\right)}=\frac{\partial{\overline f}}{\partial \overline z}\,,
\]
it is easy to check that
\[
\frac{\partial^2 M_f}{\partial z^2}(0)=\frac{\partial^2 f}{\partial z^2}(0)=h''(0) \quad \emph{if and only if}\quad \frac{\partial^2 M_f}{\partial
\overline{z}^2}(0)=\frac{\partial^2 f}{\partial \overline{z}^2}(0)=\overline{g''(0)}\,.
\]
Note also that, being $M_f$ and $f$ harmonic mappings, the equations $$\frac{\partial^2 M_f}{\partial z\partial\overline z}(0)=\frac{\partial^2 f}{\partial z\partial\overline z}(0)=0$$ always hold.
\par
We define the \emph{deviation of a sense-preserving harmonic function $f$ from their best harmonic M\"{o}bius transformation} as the function $F_f=(M_f)^{-1}\circ f$.  The Taylor expansion of $F$ in terms of $z$ and $\overline z$ has the form
\begin{eqnarray*}\label{eq-Taylordevelopment}
F_f(z) & = & z-\frac{1}{2!}\left( \frac{\overline{\omega(0)}\omega'(0)}{1-|\omega(0)|^2}\right)\,z^2+\frac{1}{2!}\left(  \frac{\overline{h'(0)\omega'(0)}}{h'(0)(1-|\omega(0)|^2)}\right)\overline z^2\\
&+& \nonumber
\frac{1}{3!}\left(Sh(0)+\frac{\overline{\omega(0)}}{1-|\omega(0)|^2}\left(\omega'(0)\frac{h''}{h'}(0)
-\omega''(0)\right)\right) z^3\\
&-&\nonumber
\frac{1}{3!}\left(\frac{h''(0)\overline{h'(0)\omega'(0)}}{(h'(0))^2(1-|\omega(0)|^2)}\right)z\overline
z^2\\
&-&\nonumber
\frac{1}{3!}\left(\frac{1}{1-|\omega(0)|^2}\left(\,\overline{\omega''(0)+2\omega'(0)\frac{h''}{h'}(0)}\,\right)\right)\overline
z^3+ \cdots
\end{eqnarray*}
\par
For a point $w\in\Omega\setminus\{0\}$, we consider the function $f_w=f(w+z)$ which is well defined in the domain $\Omega_w=\{z\in\C\colon z-w\in\Omega\}$ as before, and we define \emph{the Schwarzian derivative $S_f$ of $f$ at $w\in\Omega$} as $6$ times the difference between the coefficient of $z^3$ and the square of the coefficient of $z^2$ in the Taylor series of $F_w=(M_{f_w})^{-1}\circ f_w$. That is,
\begin{equation}\label{eq-newSchwarzian}
S_f=Sh+\frac{\overline \omega}{1-|\omega|^2}\left(\frac{h''}{h'}\,\omega'-\omega''\right)
-\frac 32\left(\frac{\omega'\,\overline
\omega}{1-|\omega|^2}\right)^2\quad \emph{in } \Omega\,,
\end{equation}
which is a generalization of the classical Schwarzian derivative
(\ref{eq-def-Sch-analytic}). (Note that, for $f$ analytic,
$\omega\equiv 0$.)

%%%%%%%%%%%%%%%%%%%%%%%%%%%%%%%%%%%%%%%%%%%%%%%%%%%%%%%%%%%%%%%%%%%%%%%%%%%%%%%%%%%%%%%%%%%%%%%%%%%%%%
\subsection{The classical Schwarzian derivative in terms of the Jacobian}\label{ssec-Jacobian}
%%%%%%%%%%%%%%%%%%%%%%%%%%%%%%%%%%%%%%%%%%%%%%%%%%%%%%%%%%%%%%%%%%%%%%%%%%%%%%%%%%%%%%%%%%%%%%%%%%%%%%

The Jacobian $J_f$ of a locally univalent harmonic mapping
$f=h+\overline g$ in the simply connected domain $\Omega$ is defined
by $J_f=|h'|^2-|g'|^2$. When $f$ is analytic (that is, when $g\equiv
0$ in $\Omega$), $J_f=|h^\prime|^2$. In this case, the classical
formulas for the pre-Schwarzian and the Schwarzian derivatives
((\ref{eq-def-preSch-analytic}) and (\ref{eq-def-Sch-analytic}),
respectively) can be written as
\begin{equation}\label{eq-ps}
Pf= \rho_z \quad \emph{and}\quad Sf=\rho_{zz}-\frac 12 \rho_z^2\,,
\end{equation}
where $\rho=\log(J_f)=\log |h^\prime|^2$ and $\rho_z=\partial \rho/\partial z$.
\par
It is just a matter of a straightforward calculation to check that for a sense-preserving harmonic mapping $f=h+\overline g$ with Jacobian $J_f$, the formula~(\ref{eq-newSchwarzian}) for the Schwarzian derivative of $f$ equals
\begin{equation}\label{eq-schw-jacobian}
S_f=\delta_{zz}-\frac 12 \delta_z^2\,,
\end{equation}
where $\delta=\log(J_f)=\log \left(|h^\prime|^2-|g^\prime|^2\right)$
in this case. In other words, we see again that $S_f$ generalizes
the classical formula (\ref{eq-def-Sch-analytic}).
\par
In view of (\ref{eq-schw-jacobian}) and (\ref{eq-ps}), we define the
\emph{pre-Schwarzian derivative} $P_f$ of a sense-preserving
harmonic mapping $f$ in $\Omega$ by $P_f=(\log(J_f))_z$. Then, we
get that
\[
S_f=(P_f)_z-\frac 12 (P_f)^2\,.
\]
\par
In terms of the cannonical decomposition of $f=h+\overline g$, we have
\begin{equation}\label{eq-preschw-jacobian}
P_f=\frac{h''}{h^\prime}-\frac{\overline\omega \omega^\prime}{(1-|\omega|^2)}\,,\quad z\in\Omega\,,
\end{equation}
where $\omega=g^\prime/h^\prime$ is the dilatation of $f$.
\par\smallskip
Notice that using the definition $P_f=(\log(J_f))_z$ for the pre-Schwarzian derivative of a sense-preserving harmonic mapping $f$,
 it is easy to conclude that $P_{\overline f}=P_f$ (just recall that $f$ is sense-preserving if and only if $\overline f$
 is sense-reversing and check that $J_{\overline f}=-J_f$). This readily implies that both the pre-Schwarzian derivative
 (\ref{eq-preschw-jacobian}) and the Schwarzian derivative (\ref{eq-newSchwarzian}) are well-defined for a\emph{ny locally univalent
 harmonic mapping} (sense-preserving or sense-reversing) in a simply connected domain $\Omega\subset\C$.

%%%%%%%%%%%%%%%%%%%%%%%%%%%%%%%%%%%%%%%%%%%%%%%%%%%%%%%%%%%%%%%%%%%%%%%%%%%%%%%%%%%%%%%%%%%%%%%%%%%%%%
\subsection{The Schwarzian derivative as a Schwarzian tensor}\label{ssec-tensor}
%%%%%%%%%%%%%%%%%%%%%%%%%%%%%%%%%%%%%%%%%%%%%%%%%%%%%%%%%%%%%%%%%%%%%%%%%%%%%%%%%%%%%%%%%%%%%%%%%%%%%%
As it was mentioned in the introduction, the Schwarzian derivative $\mathbb{S}$ presented in \cite{CHDO-DefSchw} is defined by the formula
\[
\mathbb{S}f=2(\sigma_{zz}-\sigma_z^2)\,,
\]
where $\sigma=\log(|h^\prime|+|g^\prime|)$. This means that
$\mathbb{S}f$ is equal to the Schwarzian tensor
$B_{\textbf{g}_\textbf{0}}(\log{\sigma})$ of the conformal metric
$\textbf{g}=e^{2\sigma} \textbf{g}_\textbf{0}$ relative to the
Euclidean metric $\textbf{g}_\textbf{0}$.
\par\smallskip
Consider now the metric
\[
\textbf{h}=e^{2\lambda_f} \textbf{g}_\textbf{0},
\]
with $\lambda_f=\sqrt{J_f}\,.$
\par
Note that since $f$ is sense-preserving, its Jacobian does not vanish. Bearing in mind (\ref{eq-schw-jacobian}) and the definition of $\lambda_f$, it is easy to check that
\[
S_f=2\left[(\log\lambda_f)_{zz}-(\log\lambda_f)_z^2\right]\,.
\]
Therefore, $S_f$ is equal to the Schwarzian tensor
$B_{\textbf{g}_\textbf{0}}(\log{\lambda_f})$ of the conformal metric
$e^{2\lambda_f}|dz|$ relative to the Euclidean metric.
% xxxxxxxxxxxxxxxxxxxxxxxxxxxxxxxxxxxxxxxxxxxxxxxxxxxxxxxxxxxxxxxxxxxxxxxxxxxxxxxxxxxxxxxxxxxxxxxxxxxxxxxxxxxxxxxxxxxxxx
\section{Some properties of the new pre-Schwarzian and Schwarzian derivatives}\label{sec-propertiesSchwarzian}
% xxxxxxxxxxxxxxxxxxxxxxxxxxxxxxxxxxxxxxxxxxxxxxxxxxxxxxxxxxxxxxxxxxxxxxxxxxxxxxxxxxxxxxxxxxxxxxxxxxxxxxxxxxxxxxxxxxxxxxxxxxxxxxxx
\par
In this section we obtain several properties related to the
pre-Schwarzian (\ref{eq-preschw-jacobian}) and Schwarzian
(\ref{eq-newSchwarzian}) derivatives for locally univalent harmonic
mappings in a simply connected domain. As we have mentioned in
Section~\ref{ssec-Jacobian}, there is no loss of generality by
assuming that $f$ is sense-preserving (because $P_{\overline f}=P_f$
and $S_{\overline f}=S_f$). We will consider this kind of mappings
to state the results that can be found in this section. We would
like to single out that some of these results can be proved using
the techniques developed in \cite{CHDO-DefSchw} and involving the
relationship of $S_f$ with the corresponding Schwarzian tensor.  The
proofs we show here involve only analytical arguments.
%%%%%%%%%%%%%%%%%%%%%%%%%%%%%%%%%%%%%%%%%%%%%%%%%%%%%%%%%%%%%%%%%%%%%%%%%%%%%
\subsection{The chain rule}\label{ssec-chainrule}
%%%%%%%%%%%%%%%%%%%%%%%%%%%%%%%%%%%%%%%%%%%%%%%%%%%%%%%%%%%%%%%%%%%%%%%%%%%%%
Let $f$ be a sense-preserving harmonic mapping in a simply connected
domain $\Omega\subset \C$. It is well-known that if $\varphi$ is a
locally univalent analytic function for which the composition
$f\circ\varphi$ is defined, then the function $f\circ\varphi$ is
again a sense-preserving harmonic mapping. A straightforward
calculation shows that
\begin{equation}\label{eq-chainruleharmonic}
S_{(f\circ\varphi)}=(S_f\circ\varphi)(\varphi^\prime)^2+S\varphi\,,
\end{equation}
which is a direct generalization of the chain rule
(\ref{eq-chainrule}) for the Schwarzian derivative of analytic
functions. The
same formula (\ref{eq-chainruleharmonic}) holds for the Schwarzian
derivative $\mathbb{S}$ defined by (\ref{eq-oldSch}). It is also easy to check that
$P_{f\circ\varphi}=(P_f\circ\varphi)\varphi^\prime+P\varphi$.

\par\smallskip
If $A$ is an affine mapping of the form $A(z)=az+b\overline z+c$
with $a, b, c \in\C$, the composition $A\circ f$ is harmonic (and
sense-preserving if $a\neq 0$ and $|b/a|<1$). In this case, and as
was mentioned in \cite{CHDO-DefSchw}, the Schwarzian derivative
(\ref{eq-oldSch}) of $A\circ f$ is not equal to that of $f$ unless
$b=0$.
\par
In Proposition~\ref{prop-afinne} below we prove that the
pre-Schwarzian and the Schwarzian derivatives of a locally univalent
harmonic mapping $f$ are preserved by pre-composi\-tion with affine
mappings. The proof will make use of the following lemma.
\begin{lem}\label{lem-defalternativa}
Let $f=h+\overline g$ be a sense-preserving harmonic mapping in the
simply connected domain $\Omega$ with dilatation $\omega$. Then, for
all $z_0\in\Omega$,
\begin{equation}\label{eq-lemdefalternativa}
P_f(z_0)= P(f-\overline{\omega(z_0)}g)(z_0)\quad \emph{and}\quad S_f(z_0)=S(f-\overline{\omega(z_0)}g)(z_0)\,.
\end{equation}
\end{lem}
\begin{proof}
Recall that for $z_0\in\Omega$,
\[
P_f(z_0)= Ph(z_0)-\displaystyle\frac{\overline{\omega(z_0)}\omega^\prime(z_0)}{1-|\omega(z_0)|^2}
\]
and
\[
S_f(z_0)=Sh(z_0)
+\displaystyle\frac{\overline{\omega(z_0)}}{1-|\omega(z_0)|^2}\left(\frac{h''}{h'}(z_0)\,\omega'(z_0)-\omega''(z_0)\right)
-\frac 32\left(\frac{\omega'(z_0)\,\overline{
\omega(z_0)}}{1-|\omega(z_0)|^2}\right)^2\,.
\]
\par
For each $|\lambda|<1$, define $\varphi_\lambda=h+\lambda g$ in $\Omega$. A straightforward calculation
shows that
\begin{equation}\label{eq-thm-analyticpreschwarzian5}
P\varphi_\lambda=Ph+\frac{\lambda\omega^\prime}{1+\lambda \omega}
\end{equation}
and
\begin{equation}\label{eq-thm-analyticschwarzian6}
S\varphi_\lambda=Sh-\frac{\lambda}{1+\lambda \omega}\left(\frac{h''}{h'}\omega'-\omega''\right)
-\frac 32\left(\frac{\lambda \omega'}{1+\lambda\omega}\right)^2\,.
\end{equation}
Set $\lambda=-\overline{\omega(z_0)}$ and evaluate (\ref{eq-thm-analyticpreschwarzian5}) and (\ref{eq-thm-analyticschwarzian6}) at $z_0$ to get (\ref{eq-lemdefalternativa})\,.
\end{proof}
\begin{prop}\label{prop-afinne} Let $f=h+\overline g$ be a sense-preserving harmonic mapping in $\Omega$ with dilatation $\omega_f$. Consider the
affine harmonic mapping $A(z)=az+b\overline z+c$ with $a, b, c \in
\C$, and $|a|\neq |b|$. Then, $P_{(A\circ f)}\equiv P_f$ and
$S_{(A\circ f)}\equiv S_f \emph{ in } \Omega\,.$
\end{prop}

\begin{proof} As was mentioned in Section~\ref{sec-definitionSchwarzian}, $P_{\overline f}=P_f$. Since
$\overline{A\circ f}=\overline A \circ f$, there is no loss of
generality if we assume that $A$ is sense-preserving. Moreover, it
is easy to see that $P_{(c_1 f+c_2)}=P_f$ for all sense-preserving
harmonic mappings $f$, and all complex numbers $c_1$ and $c_2$ with
$c_1\in\C\setminus\{0\}$. Hence, we can suppose that
$A(z)=z+\alpha\overline z$ for some non-zero complex number
$|\alpha|<1$. Under these assumptions, $F=A\circ f= f+\alpha
\overline f$, and what we have to prove is that $P_F=P_f$ and
$S_f=S_F$.
\par
Notice that the canonical representation of $F$ in $\Omega$ equals $H+\overline G$ where $H=h+\alpha g$ and $G=\overline\alpha h +g$. The dilatation $\omega_F$ of $F$ is
\[
\omega_F(z)=\frac{\omega_f(z)+\overline \alpha}{1+\alpha\omega_f(z)}\,.
\]
\par
Fix an arbitrary point $z_0\in\Omega$. By Lemma~\ref{lem-defalternativa},
\[
P_f(z_0)=P(h-\overline{\omega_f(z_0)}g)(z_0)\quad \emph{ and }\quad P_F(z_0)=P(H-\overline{\omega_F(z_0)}G)(z_0)\,.
\]
Since
\begin{eqnarray}
\nonumber H-\overline{\omega_F(z_0)}G &=& (h+\alpha g)-\overline{\omega_F(z_0)}(\overline \alpha h+g)\\ &=&\nonumber \left(1-\overline{\alpha\omega_F(z_0)}\right)\cdot \left(h+\frac{\alpha-\overline{\omega_F(z_0)}}{1-\overline{\alpha\omega_F(z_0)}}g\right)\\
&=& \nonumber \left(1-\overline{\alpha\omega_F(z_0)}\right)\cdot \left(h-\overline{\omega_f(z_0)}g\right)\,,
\end{eqnarray}
we obtain that $P_F(z_0)=P_f(z_0)$ for each $z_0\in\Omega$. The same
proof works to show that $S_F(z_0)=S_f(z_0)$.
\end{proof}
\par
We would like to stress the following property for the pre-Schwarzian and Sch\-warz\-ian derivatives of harmonic mappings which follows from Proposition~\ref{prop-afinne}. It allows us to re-write $P_f$ and $S_f$ at each point $z_0\in\Omega$ in terms of the (classical) pre-Schwarzian and Schwarzian derivatives of the analytic part of another harmonic function.
\begin{cor}\label{cor-relw/analyticpart}
Let $f=h+\overline g$ be a sense-preserving harmonic mapping in
$\Omega$. Then, for each $z_0\in\Omega$,
there exists a sense-preserving harmonic mapping $F=H+\overline G$
in $\Omega$ such that $$P_f(z_0)=P_F(z_0)=PH(z_0) \quad
\emph{and}\quad S_f(z_0)=S_F(z_0)=SH(z_0)\,.$$
\end{cor}
\begin{proof}
Consider the harmonic mapping $F=f-\overline{\omega_f(z_0) f}=H+\overline G\,,$ where $\omega_f$ is the dilatation of $f$. By Proposition~\ref{prop-afinne}, $P_f(z_0)=P_F(z_0)$ and $S_f(z_0)=S_F(z_0)$. Since the dilatation $\omega_F$ of $F$ vanishes at $z_0$, we get $P_F(z_0)=PH(z_0)$ and $S_F(z_0)=SH(z_0)$.
\end{proof}
%%%%%%%%%%%%%%%%%%%%%%%%%%%%%%%%%%%%%%%%%%%%%%%%%%%%%%%%%%%%%%%%%%%%%%%%%%%%%%%%%%%%%%%
\subsection{Analytic Schwarzian derivative}\label{ssec-analytic Schwarzian derivative}
%%%%%%%%%%%%%%%%%%%%%%%%%%%%%%%%%%%%%%%%%%%%%%%%%%%%%%%%%%%%%%%%%%%%%%%%%%%%%%%%%%%%%%%
If a sense-preserving harmonic mapping $f$ in the simply connected domain $\Omega$ has constant dilatation, then $f=h+\alpha \overline h+\gamma$ for some complex constants $\alpha$ and $\gamma$ (with $|\alpha|<1$) and some locally univalent function $h$ that is analytic in $\Omega$. In this case, the pre-Schwarzian derivative $P_f$ of $f$ is analytic. On the other hand, if we assume that the pre-Schwarzian derivative $P_f$ of a harmonic mapping $f=h+\overline g$ with dilatation $\omega$ is analytic, we get that
\[
\frac{\partial P_f}{\partial\overline
z}=\frac{|\omega^\prime|^2}{(1-|\omega|^2)^2}\equiv 0 \quad
\emph{in}\quad \Omega\,,
\]
which implies that $\omega$ is constant. In other words, $P_f$ is
analytic if and only if the dilatation of $f$ is constant. Moreover,
it is easy to check that $P_f\equiv 0$ in $\Omega$ if and only if
$f$ is an affine harmonic mapping. This is, $f=az+b\overline z+c$
for certain constants $a, b,$ and $c$.
\par\smallskip
The next theorem characterizes the harmonic mappings with analytic
Schwarzian derivative.
\begin{thm}\label{thm-analyticSchwarzian}
Let $f=h+\overline g$ be a sense-preserving harmonic mapping in
$\Omega$. Then, $S_f$ is analytic if and only $f$ has constant
dilatation. That is, if and only if $f=h+\alpha \overline h+\gamma$,
for some locally univalent analytic mapping $h$, some $|\alpha|<1$,
and some $\gamma\in\C$.
\end{thm}
\begin{proof} If $f=h+\alpha \overline h+\gamma$ for some locally univalent analytic mapping $h$, some $|\alpha|<1$, and some $\gamma\in\C$, then $S_f\equiv Sh$ in $\Omega$ and the result follows.
\par
Suppose that a sense-preserving harmonic mapping $f=h+\overline g$ with dilatation $\omega$ has analytic Schwarzian $S_f$ defined by
\[
S_f=Sh+\frac{\overline \omega}{1-|\omega|^2}\left(\frac{h''}{h'}\,\omega'-\omega''\right)
-\frac 32\left(\frac{\omega'\,\overline
\omega}{1-|\omega|^2}\right)^2\,.
\]
\par
If $\omega\equiv a$, $a\in \D$, then $f=h+a \overline h+\gamma$ for some locally univalent analytic function $h$ and we are done. In order to get a contradiction, assume that $\omega$ is not constant in $\Omega$. After multiplying the last equation by $(1-|\omega|^2)^2$, we obtain
\begin{equation}\label{eq-thm-analyticschwarzian0}
(Sh-S_f)(1-|\omega|^2)^2+\overline \omega(1-|\omega|^2)\left(\frac{h^{\prime\prime}}{h^\prime}\,\omega^\prime-\omega^{\prime\prime}\right)
-\frac 32{\omega^\prime}^2\overline \omega^2=0\,.
\end{equation}
Re-writing (\ref{eq-thm-analyticschwarzian0}) in terms of $1, \overline\omega$, and $\overline{\omega}^2$, and denoting
$u=Sh-S_f$, we have
\begin{equation}\label{eq-thm-analyticschwarzian1}
u+\overline\omega\left(\frac{h^{\prime\prime}}{h^\prime}\,\omega^\prime-\omega^{\prime\prime}-2\omega u\right)+
\overline{\omega}^2\left(
\omega^2u-\omega\left(\frac{h^{\prime\prime}}{h^\prime}\,\omega^\prime-\omega^{\prime\prime}\right)-\frac
32{\omega^\prime}^2\right)=0\,.
\end{equation}
Differentiate with respect to $\overline z$ in (\ref{eq-thm-analyticschwarzian1}) to get

\begin{equation}\label{eq-thm-analyticschwarzian2}
\overline{\omega^\prime}\left(\frac{h^{\prime\prime}}{h^\prime}\,\omega^\prime-\omega^{\prime\prime}-2\omega u\right)+
2\overline{\omega\omega^\prime}\left(
\omega^2u-\omega\left(\frac{h^{\prime\prime}}{h^\prime}\,\omega^\prime-\omega^{\prime\prime}\right)-\frac
32{\omega^\prime}^2\right)=0\,.
\end{equation}
\par
Now, since $\omega$ is not constant, there exists a disk $D(z_0,R)$ with center $z_0$ and radius $R>0$ contained in $\Omega$ where $\omega^\prime\neq 0$. We can divide both sides of (\ref{eq-thm-analyticschwarzian2}) by $\overline{\omega^\prime}$, and take derivatives with respect to $\overline z$ again to obtain
\[
\overline{\omega^\prime}\left(
\omega^2u-\omega\left(\frac{h^{\prime\prime}}{h^\prime}\,\omega^\prime-\omega^{\prime\prime}\right)-\frac
32{\omega^\prime}^2\right)=0\, \quad \emph{in } D(z_0,R)\,.
\]
Keeping in mind that $\omega'\neq 0$ in $D(z_0,R)$, we see that
\begin{equation}\label{eq-thm-analyticschwarzian3}
\omega^2u-\omega\left(\frac{h^{\prime\prime}}{h^\prime}\,\omega^\prime-\omega^{\prime\prime}\right)-\frac
32{\omega^\prime}^2\equiv 0\,.
\end{equation}
Hence, from (\ref{eq-thm-analyticschwarzian2}) we have
\begin{equation}\label{eq-thm-analyticschwarzian4}
\frac{h''}{h'}\,\omega^\prime-\omega^{\prime\prime}-2\omega u=0\,,
\end{equation}
which implies, by (\ref{eq-thm-analyticschwarzian1}), that  $u=Sh-S_f\equiv
0$ in $D(z_0,R)\subset \Omega$. But in this case, using (\ref{eq-thm-analyticschwarzian4}), we obtain
\[
\frac{h^{\prime\prime}}{h^\prime}\,\omega^\prime-\omega^{\prime\prime}=0\,,
\]
and substituting in (\ref{eq-thm-analyticschwarzian3}), we get $\omega^\prime\equiv 0$ in $D(z_0,R)$. This is a contradiction. Hence, $\omega\equiv \alpha\in\D$ and $f=h+\alpha\overline h+\gamma$, as we wanted to prove.
\end{proof}
\par
\par\smallskip
Recall that a locally univalent analytic function $T$ in the simply
connected domain $\Omega$ is a M\"{o}bius transformation if and only
if $Sf\equiv 0$ in $\Omega$. A locally univalent harmonic M\"{o}bius
transformation is a harmonic mapping of the form $f=\alpha T+\beta
\overline T+\gamma$, where $T$ is a (classical) M\"{o}bius
transformation and $\alpha,\beta, \gamma\in\C$ are constants. These
transformations were characterized in \cite{CHDO-DefSchw} in terms
of the Schwarzian derivative (\ref{eq-oldSch}). The following
corollary also characterizes the harmonic M\"{o}bius transformation
in terms of $S_f$.

\begin{cor} A locally univalent mapping $f=h+\overline g$ equals a
harmonic M\"{o}bius transformation if and only if $S_f\equiv 0$.
\end{cor}

\begin{proof} Since $S_{\overline f}=S_f$, we can assume that $f$ is sense-preserving. If $f$ is a M\"{o}bius harmonic transformation, then $S_f=0$. Conversely, if $S_f=0$, by Theorem~\ref{thm-analyticSchwarzian}, $f=h+\alpha \overline h+\gamma$ for some constants $\alpha\in\D$ and $\gamma\in\C$, and for some locally univalent analytic function $h$ with $Sh=0$. Hence, $h$ is a M\"{o}bius transformation.
\end{proof}
%%%%%%%%%%%%%%%%%%%%%%%%%%%%%%%%%%%%%%%%%%%%%%%%%%%%%%%%%%%%%%%%%%%%%%%%%%%%%%%%%%%%%%%
\subsection{Harmonic Schwarzian derivative}\label{ssec-antianalytic Schwarzian derivative}
%%%%%%%%%%%%%%%%%%%%%%%%%%%%%%%%%%%%%%%%%%%%%%%%%%%%%%%%%%%%%%%%%%%%%%%%%%%%%%%%%%%%%%%
The (classical) Schwarzian derivative of a locally univalent analytic function $\varphi$ is also analytic. This can be understood as a consequence of the fact that for such a function $\varphi$, its pre-Schwarzian derivative $P\varphi$ is an analytic function.
\par
In the case of harmonic mappings $f=h+\overline g$, it is not difficult to check that the pre-Schwarzian derivative $P_f$ is harmonic if and only if the dilatation $\omega$ of $f$ is constant (which is also equivalent, as we proved in Section~\ref{ssec-analytic Schwarzian derivative}, to the fact that the pre-Schwarzian and the Schwarzian derivatives of $f$ are analytic). In this section we see that this is precisely the case for which $S_f$ is harmonic.
\par
\begin{thm}\label{thm-antianalyticSchwarzian}
Let $f=h+\overline g$ be a sense-preserving harmonic mapping in a simply connected domain $\Omega$. If $S_f$ is harmonic, then  $S_f$ is analytic.
\end{thm}
\begin{proof}
By a straightforward calculation,
\begin{eqnarray}\label{eq-thm-antianalyticSchwarzian1}
\frac{\partial ^2 S_f}{\partial{\overline z}\partial z} &=& \varphi_2\frac{2\overline\omega\overline{\omega^\prime}}{(1-|\omega|^2)^3}+
\varphi_1^\prime\frac{\overline{\omega^\prime}}{(1-|\omega|^2)^2}
- 9{\omega^\prime}^3\frac{{\overline\omega}^2\overline{\omega^\prime}}{(1-|\omega|^2)^4}\,,
\end{eqnarray}
where
\[
\varphi_1=\omega^\prime\frac{h^{\prime\prime}}{h^\prime}-\omega^{\prime\prime}\quad \emph{and}\quad \varphi_2=\varphi_1\omega^\prime -3\omega^\prime\omega^{\prime\prime}\,.
\]
\par
Assume that $\Delta S_f\equiv 0$ in $\Omega$. Hence, according to
formula (\ref{eq-laplacian}), the right hand side of
(\ref{eq-thm-antianalyticSchwarzian1}) is identically zero in
$\Omega$. After multiplying (\ref{eq-thm-antianalyticSchwarzian1})
by $(1-|\omega|^2)^4$, we get
\begin{eqnarray}\label{eq-thm-antianalyticSchwarzian2}
\quad\quad 0 &\equiv& 2\varphi_2\overline\omega \overline{\omega^\prime}(1-|\omega|^2)+
\varphi_1^\prime \overline{\omega^\prime}(1-|\omega|^2)^2-9{\omega^\prime}^3{\overline\omega}^2\overline{\omega^\prime}\,.
\end{eqnarray}
If the dilatation $\omega$ of $f$ is constant, then the Schwarzian
derivative of $f$ is analytic. If $\omega$ is not constant, there
exists a disk $D(z_0,R)\subset \Omega$ where $\omega^\prime\neq 0$.
Dividing (\ref{eq-thm-antianalyticSchwarzian2}) by
$\overline{\omega^\prime}$, we obtain
\begin{eqnarray}
\nonumber \quad\quad 0 &\equiv& 2\varphi_2\overline\omega (1-|\omega|^2)+
\varphi_1^\prime (1-|\omega|^2)^2-9{\omega^\prime}^3{\overline\omega}^2\,,
\end{eqnarray}
which is equivalent to
\begin{eqnarray}\label{eq-thm-antianalyticSchwarzian4}
\quad\quad  0 &\equiv& \varphi_1^\prime+\left(2\varphi_2-2\varphi_1^\prime\omega\right)\overline\omega+
\left(\varphi_1^\prime{\omega}^2-2\varphi_2\omega-9{\omega^\prime}^3\right){\overline\omega}^2\,.
\end{eqnarray}

Taking derivatives with respect to $\overline z$ in
(\ref{eq-thm-antianalyticSchwarzian4}) and dividing by
$\overline{\omega^\prime}$, we have
\begin{eqnarray*}\label{eq-thm-antianalyticSchwarzian5}
\qquad\qquad  0 &\equiv& \left(2\varphi_2-2\varphi_1^\prime\omega\right)+2
\left(\varphi_1^\prime{\omega}^2-2\varphi_2\omega-9{\omega^\prime}^3\right){\overline\omega}\, \quad\emph{in}\quad D(z_0,R)\,.
\end{eqnarray*}
If $\varphi_1^\prime{\omega}^2-2\varphi_2\omega-9{\omega^\prime}^3\neq 0$, then $\overline\omega$ must be constant (in this case $\omega$ would be both analytic and anti-analytic). This is a contradiction since we are assuming that $\omega^\prime\neq 0$ in $D(z_0,R)$. Hence, $\varphi_1^\prime{\omega}^2-2\varphi_2\omega-9{\omega^\prime}^3\equiv 0$ (in some disk $D\subset D(z_0,R)$). This fact implies that $\varphi_2-\varphi_1^\prime\omega=0$ in $D$. Therefore, using  (\ref{eq-thm-antianalyticSchwarzian4}), we obtain that  $\varphi_2=\varphi_1^\prime=0$ as well and we conclude that $\omega$ is constant in $D$. This proves the theorem.
\end{proof}
\par
%%%%%%%%%%%%%%%%%%%%%%%%%%%%%%%%%%%%%%%%%%%%%%%%%%%%%%%%%%%%%%%%%%%%%%%%%%%%%%%%%%%%%%%%%%%%%%%
\section{Locally univalent harmonic mappings with equal pre-Schwarzian derivatives}\label{sec-equalSchwarzian}
%%%%%%%%%%%%%%%%%%%%%%%%%%%%%%%%%%%%%%%%%%%%%%%%%%%%%%%%%%%%%%%%%%%%%%%%%%%%%%%%%%%%%%%%%%%%%%%
By exploiting the differential geometry of the associated minimal surface to a locally univalent, sense-preserving harmonic function $f=h+\overline g$ with dilatation $\omega=q^2$ (for some analytic function $q$ with $|q|<1$), Chuaqui, Duren, and Osgood characterized the sense-preserving harmonic functions with equal Schwarzian derivatives $\mathbb{S}$ (see \cite{CHDO-DefSchw}). Their theorem can be stated as follows.
\begin{other}\label{other}
Let $f=h+\overline g$ and $F=H+\overline G$ be sense-preserving harmonic functions defined on a common domain $\Omega\subset \C$. Let $\omega_f=q^2$ and $\omega_F=Q^2$ be the dilatations of $f$ and $F$, respectively. Then,
\begin{itemize}
\item[(a)] If $\omega_f$ is not constant, then $\mathbb{S}f=\mathbb{S}F$ if and only if  $|h^\prime|+|g^\prime|=c\ (|H^\prime|+|G^\prime|)$ for some constant $c>0$\,.
\item[(b)]  If $\omega_f$ is constant, then $\mathbb{S}f=\mathbb{S}F$ if and only if $f=h+\alpha \overline h$, $F=H+\beta \overline H$, and $H=T(h)$ for some locally univalent functions $h$ and $H$, some complex constants $\alpha$ and $\beta$ with $|\alpha|<1$ and $|\beta|<1$, and some analytic M\"{o}bius transformation $T$.
\end{itemize}
\end{other}
\par
Using Theorem~\ref{thm-analyticSchwarzian}, it is easy to prove that
an analogous result to (b) in the last theorem holds for the new
Schwarzian derivative (\ref{eq-newSchwarzian}). Namely, we can see
that if $\omega_f$ is constant, then $S_f=S_F$ if and only if
$\mathbb{S}f=\mathbb{S}F$. Hence, the functions $f$ and $F$ are
related as in statement (b) in Theorem~\ref{other}.
\par
Keeping in mind that $\mathbb{S}f$ is defined in terms of $\sigma=|h^\prime|+|g^\prime|$, and that the definition of $S_f$ involves the Jacobian
$J_f=|h^\prime|^2-|g^\prime|^2$, it seems logical to expect that the condition $J_f=c J_F$ (for some constant $c>0$) should have some influences on the relation between the Schwarzian derivatives of $f$ and $F$. We have not been able to prove that if $S_f=S_F$ (and the dilatations are non-constant), then the Jacobians of $f$ and $F$ are homothetic. However, in the next theorem, we see that this is true for the \emph{pre-Schwarzian} derivative.
\begin{thm}\label{thm-jacobians}
Let $f=h+\overline g$ and $F=H+\overline G$ be sense-preserving
harmonic functions defined on a common simply connected domain
$\Omega\subset \C$ with non-constant dilatations $\omega_f$ and
$\omega_F$, respectively. Then $P_f=P_F$ if and only if the
Jacobians are homothetic. That is, if and only if
$|h^\prime|^2-|g^\prime|^2=c\ (|H^\prime|^2-|G^\prime|^2)$ for some
constant $c>0$\,.
\end{thm}
\begin{pf}
Recall that $P_f=(\log J_f)_z$. Therefore, if the Jacobians are
homothetic, then $P_f=P_F$.
\par
Assume now that $P_f=P_F$. Since the dilatations are not constant,
there exists a point $z_0\in\Omega$ with $\omega_f^\prime(z_0)$ and
$\omega_F^\prime(z_0)$ different from $0$. Consider a Riemann
mapping $\psi$ from the disk onto $\Omega$ with $\psi(0)=z_0$ and
define $f_0=f\circ\psi=h_0+\overline{g_0}$ and
$F_0=F\circ\psi=H_0+\overline{G_0}$. The dilatations of $f_0$ and
$F_0$ are $\omega_{f_0}=\omega_f\circ \psi$ and
$\omega_{F_0}=\omega_F\circ \psi$, respectively. Hence, they are
non-constant analytic functions in the unit disk with
$\omega_{F_0}^\prime(0)\neq 0$ and $\omega_{f_0}^\prime(0)\neq 0$.
\par
Using that $P_{(f\circ\psi)}=(P_f\circ\psi)\psi^\prime+P\psi$, we
see that that $P_{f_0}=P_{F_0}$ in $\D$.
\par\smallskip
Now, it is a straightforward calculation to check that the
sense-preserving harmonic mappings
$f_1=f_0-\overline{\omega_{f_0}(0)\, f_0} =h_1+\overline{g_1}$ and
$F_1=F_0-\overline{\omega_{F_0}(0)\,f_0}=H_1+\overline{G_1}$ have
equal pre-Schwarzian derivatives and that their corresponding
dilatations $\omega_{f_1}$ and $\omega_{F_1}$ fix the origin and
have non-zero derivatives at $z=0$. Notice also that the functions
$f_2=(f_1-f_1(0))/h_1^\prime(0)$ and
$F_2=(F_1-F_1(0))/H_1^\prime(0)$ satisfy the additional
normalizations $h_2(0)=H_2(0)=g_2(0)=G_2(0)=0$ and
$h_2^\prime(0)=H_2^\prime(0)=1$.
\par\smallskip
To sum up, we have constructed two sense-preserving harmonic
mappings (that we again denote by $f=h+\overline{g}$ and
$F=H+\overline{G}$\,) ç in the unit disk with non-constant
dilatations $\omega_{f}$ and $\omega_{F}$, respectively, with
$\omega_{f}(0)=\omega_{F}(0)=h(0)=H(0)=g(0)=G(0)=0$,
$h^\prime(0)=H^\prime(0)=1$, $\omega^\prime_{f}(0)\neq 0$, and
$\omega^\prime_{F}(0)\neq 0$, and such that $P_{f}=P_{F}$.
\par\smallskip
We first need to prove the following lemma.
\begin{lem}\label{lem-Sh}
Let $f$ and $F$ be two sense-preserving harmonic mappings normalized
as above. If $P_{f}=P_{F}$, then $Ph=PH$ and therefore $h=H$.
\end{lem}
\begin{pf}
Suppose that $P_{f}=P_{F}$. By (\ref{eq-preschw-jacobian}), this
condition is
\begin{equation}\label{eq-Sch-equal}
Ph-\frac{\overline{\omega_{f}}\,
\omega_{f}^\prime}{1-|\omega_{f}|^2}=
PH-\frac{\overline{\omega_{F}}\,
\omega_{F}^\prime}{1-|\omega_{F}|^2} \,.
\end{equation}
Since $P_{f}=P_{F}$, we have that for all non-negative integer $n$
\[
\left(\frac{\partial\,^n P_f}{\partial z^n}\right)(0)=\left(\frac{\partial\,^n P_F}{\partial z^n}\right)(0)\,.
\]
Using that $\omega_f(0)=\omega_F(0)=0$ and (\ref{eq-Sch-equal}), we
obtain
\[
\left(\frac{\partial\,^n P_f}{\partial z^n}\right)(0)=
Ph^{(n)}(0)\quad \emph{and}\quad \left(\frac{\partial\,^n
P_F}{\partial z^n}\right)(0)=PH^{(n)}(0)\,,
\]
where $Ph^{(n)}(0)$ and $PH^{(n)}(0)$ denote the $n-$th derivatives
of $Ph$ and $PH$, respectively. Since both $Ph$ and $PH$ are
analytic functions in the unit disk, it follows that $Ph=PH$ in
$\D$. Hence, $h=AH+B$ for certain constants $A, B$ with $A\neq 0$.
Using that $h(0)=H(0)=0$ and that $h^\prime(0)=H^\prime(0)$, we get that
$h\equiv H$ in the unit disk.
\end{pf}
\par\smallskip
We now continue with the proof of Theorem~\ref{thm-jacobians}. Once
we know that if $P_f=P_F$ then $h=H$, we have from
(\ref{eq-Sch-equal}) that
\begin{equation}\label{eq-Sch-equal-2}
\frac{\overline{\omega_{f}}\, \omega_{f}^\prime}{1-|\omega_{f}|^2}= \frac{\overline{\omega_{F}}\, \omega_{F}^\prime}{1-|\omega_{F}|^2}\,.
\end{equation}
\par
Taking derivatives with respect to $\overline z$ on both sides of (\ref{eq-Sch-equal-2}), we obtain
\begin{equation}\label{eq-Sch-equal-3}
\frac{\overline {\omega_f}^\prime\,
\omega_f^\prime}{(1-|\omega_f|^2)^2}= \frac{\overline
{\omega_F}^\prime\, \omega_F^\prime}{(1-|\omega_F|^2)^2}\,,
\end{equation}
which is equivalent to
\begin{equation}\label{eq-Sch-equal-4}
\overline {\omega_f}^\prime\, \omega_f^\prime
\left(1+\frac{1}{(1-|\omega_f|^2)^2}-1\right) =\overline
{\omega_F}^\prime\, \omega_F^\prime
\left(1+\frac{1}{(1-|\omega_f|^2)^2}-1\right)\,.
\end{equation}
Note that (\ref{eq-Sch-equal-4}) can be written as
\[
\overline{\omega_f^\prime} \omega_f^\prime +\overline{\omega_f} P= \overline{\omega_F^\prime} \omega_F^\prime +\overline{\omega_F} Q\,,
\]
for certain functions $P$ and $Q$. Hence,
\[
\left(\frac{\partial^n}{\partial z^n}\left(\overline{\omega_f^\prime} \omega_f^\prime +\overline{\omega_f} P\right)\right)(0)= \left(\frac{\partial^n}{\partial z^n}\left(\overline{\omega_F^\prime} \omega_F^\prime +\overline{\omega_F} Q\right)\right)(0)
\]
for all non-negative integer $n$. Using again that both $\omega_f$
and $\omega_F$ fix the origin, we obtain
\[
\overline{\omega_f^\prime(0)}\omega_f^{(n)}(0)
=\overline{\omega_F^\prime(0)}\omega_F^{(n)}(0)\,,
\]
which implies that
\[
\overline{\omega_f^\prime(0)}
\omega_f^\prime=\overline{\omega_F^\prime(0)} \omega_F^\prime\,.
\]
Therefore, $\omega_F^\prime=\overline k\omega_f^\prime$, with
$k=\omega_f^\prime(0)/\omega_F^\prime(0)\neq 0$, and
(\ref{eq-Sch-equal-3}) becomes
\[
\frac{|\omega_{f}^\prime|}{1-|\omega_{f}|^2}= |k|^2
\frac{|\omega_{f}^\prime|}{1-|\overline k\omega_{f}|^2}\quad
\]
in the unit disk. In particular, taking $z=0$, we see that $|k|=1$. Hence,
\[
J_f=|h^\prime|^2(1-|\omega_f|^2)= J_F\,,
\]
which proves the theorem for functions with the normalizations stated
before Lem\-ma~\ref{lem-Sh} (that is, for those functions $f_2$ and
$F_2$ that we constructed from the initials $f$ and $F$). Unwinding
the definitions back to $f$ and $F$ leads directly to the assertion
of the theorem.
\end{pf}
\par\medskip
We would like to remark that using the arguments from
\cite{CHDO-DefSchw}, it is possible to prove that the condition
$J_f=c\, J_F$ is equivalent to the fact that
$(H^\prime,\omega_F)=\gamma(h^\prime,\omega_f)$ for some element
$\gamma$ in the group $\mathcal G$ generated by the following
transformations of the pair $(h^\prime,\omega_f)$:
\[
R_p(\lambda):\quad (h^\prime,\omega_f)\to (\lambda h^\prime,\omega_f)\,, \quad \lambda\neq 0\,,
\]
\[
R_q(\mu):\quad (h^\prime,\omega_f)\to (h^\prime,\mu \omega_f)\,,
\quad |\mu|=1\,,
\]
\[
I(a):\quad (h^\prime,\omega_f)\to \left(\frac{(1+\overline
a\omega_f)^2}{1-|a|^2}h^\prime, \frac{a+\omega_f}{1+\overline a
\omega_f}\right)\,,\ \quad a\in\D\,.
\]
Notice that this group $\mathcal G$ is different from the one presented in \cite{CHDO-DefSchw}.
\par
\par\medskip
Except for the case when the dilatations are constant, the action of the elements of the group $\mathcal G$ can be identified with the following operators acting on the space of sense-preserving harmonic mappings: Let $f=h+\overline g$ be such a function with dilatation $\omega$. Then
\begin{equation}\label{eq-functions}
R_p(\lambda)(f)=\lambda h+\overline{\lambda g}\,,\quad R_q(\mu)(f)=h+\overline{\mu g}\,,\quad \emph{and} \quad I(a)(f)= H+\overline G\,,
\end{equation}
where $H^\prime=\lambda h^\prime/\sqrt{\phi_a^\prime\circ \omega}$
for certain constant $\lambda\neq 0$, and
$Q=G^\prime/H^\prime=\phi_a\circ \omega$. The function $\phi_a$ is
defined in the unit disk by
\begin{equation}\label{eq-automorphism}
\phi_a(z)=\frac{a+z}{1+\overline a z}\,.
\end{equation}
Note that since $f$ is supposed to be a sense-preserving harmonic
mapping, the composition $\phi_a\circ \omega$ is well defined. Also,
that $I_a(f)=f+\overline{af}$.
\par\smallskip
It is not difficult to check that the operations described by
(\ref{eq-functions}) are precisely the ones we used in the proof of
Theorem~\ref{thm-jacobians} to construct $f_2$ and $F_2$ from $f_0$
and $F_0$, respectively.
\par\smallskip
We end this section with the following result that follows directly
from the proof of Theorem~\ref{thm-jacobians}.
\begin{cor}
Let $f=h+\overline g$ and $F=H+\overline G$ be sense-preserving harmonic functions defined on a common domain $\Omega\subset \C$ with non-constant dilatations $\omega_f$ and $\omega_F$, respectively. Then $P_f=P_F$ if and only if then there exist $a\in\D$, $\mu$ on $\partial\D$, and $\lambda \neq 0$ such that
\[
\omega_F=\mu(\phi_a\circ\omega_f)\quad \emph{and}\quad
H^\prime=\frac{\lambda\, h^\prime}{\sqrt{\phi_a\circ\omega_f}}\,,
\]
where $\phi_a$ is the automorphism of $\D$ defined by
(\ref{eq-automorphism}).
\end{cor}
% xxxxxxxxxxxxxxxxxxxxxxxxxxxxxxxxxxxxxxxxxxxxxxxxxxxxxxxxxxxxxxxxxxxxxxxxxxxxxxxxxxxxxxxxxxxxxxxxxxxxxxxxxxxxxxxxxxxxxxxxxxxxxxxx
\section{Norm of the Schwarzian derivative}\label{sec-Schwarziannorm}
% xxxxxxxxxxxxxxxxxxxxxxxxxxxxxxxxxxxxxxxxxxxxxxxxxxxxxxxxxxxxxxxxxxxxxxxxxxxxxxxxxxxxxxxxxxxxxxxxxxxxxxxxxxxxxxxxxxxxxxxxxxxxxxxx
\par
Let $f=h+\overline g$ be a sense-preserving harmonic mapping in the unit disk $\D$. Using the chain rule (\ref{eq-chainruleharmonic}),
it is easy to see that for each $z\in\D$,
$$|S_f(z)|(1-|z|^2)^2=|S_{(f\circ{\phi})}(0)|\,,$$ where $\phi$ is any automorphism of the unit disk with $\phi(0)=z$.
The \emph{Schwarzian norm} $\|S_f\|$ of $f$ is defined by
\[
\|S_f\|=\sup_{z\in\D}|S_f(z)|(1-|z|^2)^2\,.
\]
\par
Krauss \cite{Krauss} proved that if $f$ is analytic and univalent in the unit disk, then  $\|Sf\|\leq 6$. Another related result due to Nehari \cite{Nehari} states that if $f$ is \emph{convex} (that is, $f$ is univalent in the unit disk and the domain $f(\D)$ is convex), then $\|Sf\|\leq 2$. Both constants $2$ and $6$ are sharp: The \emph{analytic Koebe mapping}
\[
k(z)=\frac{z}{(1-z)^2},\quad z\in\D\,,
\]
is univalent in the unit disk. It maps $\D$ onto the full plane
minus the part of the negative real axis from $-1/4$ to infinity and
has Schwarzian derivative
\[
Sk(z)=-\frac{6}{(1-z^2)^2}\,.
\]
Hence, $\|S_f\|=6$. Note that $|Sf(r)|(1-r^2)^2=6$ for all real numbers $-1<r<1$. The (analytic and univalent) strip mapping $s$, defined in the unit disk by (\ref{eq-strip-analytic}) is convex. It has Schwarzian derivative
\[
Ss(z)=\frac{2}{(1-z^2)^2}\,.
\]
The same constant $2$ also appears in the Nehari criterion for univalence. This criterion states that if the Schwarzian norm of a locally univalent analytic function $f$ in the unit disk is bounded by $2$, then $f$ is univalent in $\D$.
\par\smallskip
Let now $\mathcal F$ be the family of univalent sense-preserving
harmonic mappings $f=h+\overline g$ in the unit disk with dilatation
$w=q^2$ for some analytic function $q$ (with $|q|<1$) in $\D$.
Notice that the harmonic Koebe function $K$ as in
(\ref{eq-Koebe-harmonic}) does not belong to $\mathcal F$ since $K$ has
dilatation $w(z)=z$. There is an analogous result to that by Krauss
for the functions in $\mathcal F$ in terms of the Schwarzian
derivative $\mathbb{S}$ defined by (\ref{eq-oldSch}). Namely, it was
proved in \cite{Ch-D-O-08} that there exists a constant $C_1$ such
that
\begin{equation}\label{eq-kraussoldSchw}
\|\mathbb{S}f\|=\sup_{z\in\D}|\mathbb{S}f(z)|(1-|z|^2)^2\leq C_1
\end{equation}
for all $f\in\mathcal F$. The sharp value of $C_1$ is unknown. It is thought
to be equal to $16$ since the Schwarzian norm of the univalent
harmonic mapping $K_2$ constructed as the horizontal shear of the
Koebe function with dilatation $\omega(z)=z^2$ belongs to $\mathcal F$ and
satisfies $\|\mathbb{S}K_2\|=16$. The function $K_2=h+\overline g$
is defined by
\begin{equation}\label{eq-formulaK2}
h(z)=\frac 13\left[\frac{1}{(1-z)^3}-1\right]\ \emph{and  } g(z)=\left[\frac{z^2-z+\frac 13}{(1-z)^3}-\frac 13\right]\,.
\end{equation}
\par
We will prove an analogous bound to that in (\ref{eq-kraussoldSchw})
for the norm of the Schwarzian derivative $S_f$. Concretely, we will
see that there exists a constant (with unknown sharp value) $C_2$
such that \emph{for all} univalent harmonic mappings $f$ in the unit
disk (with no extra assumption on the dilatation), the inequality
$\|S_f\|\leq C_2$ holds. The proof will use the results obtained in
Section \ref{ssec-Schwarziannormconvex} below that are related to
convex (univalent) harmonic mappings.
\par\smallskip
All the results obtained in this section can be stated in terms of the pre-Schwarzian norm defined by
\[
\|P_f\|=\sup_{z\in\D} |P_f(z)|(1-|z|^2)\,.
\]
The proofs are similar to those presented here (and often easier). This is the reason why we have
decided to leave most of the explicit statements involving $\|P_f\|$
out. The exception is Theorem~\ref{thm-convexpres}, where we obtain
a sharp bound for the pre-Schwarzian norm of any convex harmonic
mapping.
%%%%%%%%%%%%%%%%%%%%%%%%%%%%%%%%%%%%%%%%%%%%%%%%%%%%%%%%%%%%%%%%%%%%%%%%%%%%%%%%%%%%%%%%%%%%%%%
\subsection{The Schwarzian norm of convex harmonic mappings}\label{ssec-Schwarziannormconvex}
%%%%%%%%%%%%%%%%%%%%%%%%%%%%%%%%%%%%%%%%%%%%%%%%%%%%%%%%%%%%%%%%%%%%%%%%%%%%%%%%%%%%%%%%%%%%%%%
Let $f=h+\overline g$ be a univalent harmonic mapping in the unit
disk. If $f(\D)$ is convex, we say that $f$ is a convex (harmonic)
mapping. In Section~\ref{ssec-shear}, we have presented three
particular examples of this kind of mappings. We begin this section
by computing their Schwarzian derivatives and Schwarzian norms.
\begin{ex} Consider the half-plane harmonic mapping $L$ defined by (\ref{eq-formulaL}). Its Sch\-warz\-ian derivative equals
$$\begin{array}{cll}
S_L(z)&=&\displaystyle-\frac 32\frac{1}{(1-z)^2}+\frac{3\overline
z}{(1-z)(1-|z|^2)} -\frac 32\left(\frac{\overline
z}{1-|z|^2}\right)^2\\[0.4cm]
&=& \displaystyle -\frac32\left(\frac{1}{1-z}-\frac{\overline
z}{1-|z|^2}\right)^2=-\frac32\left(\frac{1-\overline
z}{1-z}\cdot\frac{1}{1-|z|^2}\right)^2\,.\end{array}$$
Note that for all $z\in \D$, $(1-|z|^2)^2\left|S_L(z)\right|=3/2$. Therefore, $\|S_L\|=3/2$.
\end{ex}
\begin{ex} The Schwarzian derivative of the half-strip harmonic mapping $S_1$ given by (\ref{eq-formulaS1}) is
$$S_{S_1}(z)=\frac{2(1-|z|^2)(1+\overline z)+3(1-\overline z^2)(1+z)}{2(1-z^2)(1+z)(1-|z|^2)^2}\,.$$ Therefore,
$$|S_{S_1}(z)|(1-|z|^2)^2=\left|\frac{(1+\overline z)(1-|z|^2)}{(1-z^2)(1+z)}+\frac 32\cdot\frac{1-\overline z^2}{1-z^2}\right|\leq \frac 52\,.$$
Moreover, $S_{S_1}(r)(1-r^2)^2=5/2$ for all real numbers $r\in (-1,1)$ and hence, $\|S_{S_1}\|=5/2$.
\end{ex}
\par
The final example is related to the strip harmonic mapping $S_2$ defined by (\ref{eq-formulaS2}).
\begin{ex} The Schwarzian derivative of $S_2$ is
$$S_{S_2}(z)=2\ \ \frac{2 + |z|^4  - \overline z^2 -2|z|^4 \overline z^2 }{(1 - z^2)(1 - |z|^4)^2}\,.$$ Hence,
$$\begin{array}{lll}\left|S_{S_2}(z)\right| (1-|z|^2)^2 & = & 2\  \ \displaystyle
\left|\frac{2 + |z|^4  - \overline z^2 -2|z|^4 \overline z^2 }{(1 - z^2)(1 +|z|^2)^2}\right|\\[0.5cm]
&=& 2\ \ \displaystyle\left|\frac{1-|z|^4+1-\overline z^2+2|z|^4(1-\overline z^2)}{(1 - z^2)(1 +|z|^2)^2}\right|\\[0.5cm]
&\leq &2\ \ \displaystyle\frac{2+|z|^2+2|z|^4}{(1+|z|^2)^2}=4-2\frac{|z|^2}{(1+|z|^2)^2}\leq 4\,.
\end{array}$$
Since $S_{S_2}(0)=4$, we obtain that $\|S_{S_2}\|=4$.
\end{ex}
\par
As we said in the beginning of this Section, it is known that the
Schwarzian norm of every convex analytic function is less than or
equal to $2$. In the next proposition, we prove an analogous theorem
for convex harmonic mappings. To do so, we use one of the theorems
by Clunie and Sheil-Small (see \cite[Theorem 5.7]{CSS}): If $f$ is a
convex harmonic mapping in the unit disk, then for each $\varepsilon
\in \overline \D$, the analytic function
$\varphi_\varepsilon=h+\varepsilon g$ is close-to-convex in $\D$,
thus univalent.
\begin{prop}\label{prop-convex}
Let $f=h+\overline g$ be a convex harmonic mapping in the unit disk. Then
$$\|S_f\|\leq 6\,.$$
\end{prop}
\begin{proof} Fix an arbitrary point $z_0\in\D$. Since $f$ is convex, the analytic functions $\varphi_\varepsilon=h+\varepsilon g$ are univalent for all $\varepsilon\in\overline{\D}$. In particular, for $\varepsilon=-\overline{\omega(z_0)}$, where $\omega$ is the dilatation of $f$. Therefore, using Lemma~\ref{lem-defalternativa} and Krauss' Theorem, we get
$$|S_f(z_0)|(1-|z_0|^2)^2=|S(h-\overline{\omega(z_0)}g)(z_0)|(1-|z_0|^2)^2\leq 6\,.$$
\end{proof}
\par\smallskip
We do not know if the constant $6$ in the last proposition is sharp. However, by considering the pre-Schwarzian norm instead of $\|S_f\|$, we can prove the following result.
\begin{thm}\label{thm-convexpres}
Let $f=h+\overline g$ be a convex harmonic mapping in the unit disk. Then,
$$\|P_f\|=\sup_{z\in\D} |P_f(z)|(1-|z|^2)\leq 5\,.$$
The constant $5$ is sharp.
\end{thm}
\begin{proof} Without loss of generality, we can assume that $f$ is sense-preserving.
\par\smallskip
If the dilatation of $f$ is constant, then $f=ah+b \overline h+c$
for some convex analytic function $h$ and certain constants $a, b,$
and $c$, and we have $P_f=Ph$. In this case, a result due to
Yamashita \cite{Y} states that $\|Ph\|\leq 4$. Hence, the result
follows for convex harmonic functions with constant dilatations.
\par
To prove the general case, we first find the corresponding upper bound for $|P_f(0)|$.
\par
Using that $P_{(A\circ f)}=P_f$ for all affine mappings
$A(z)=z+\alpha\overline z$, we can suppose that $f\in C_H^0=\{f\in
S_H^0\colon f(\D)\emph{ is convex}\}$. It is known that if
$f=h+\overline g \in C_H^0$, then $|h''(0)|\leq 3$ (see \cite[p.
50]{Dur-Harm}). Hence, we obtain $|P_f(0)|=|Ph(0)|=|h''(0)|\leq 3$.
\par
The function $L$ defined by (\ref{eq-formulaL}) satisfies
$|P_L(0)|=3$.
\par
Now, fix an arbitrary point $z_0\in\D$ and take an automorphism
of the disk $\phi$ with $\phi(0)=z_0$. Then, the function $f\circ
\phi$ is convex and
\[
|P_f(z_0)|(1-|z_0|)=|P_{f\circ\psi}(0)-2\overline{z_0}|\leq
|P_{f\circ\phi}(0)|+2|z_0|\leq 5\,.
\]
Therefore, $\|P_f\|\leq 5$.
\par
It is easy to check that $\lim_{r\to 1^-}\|P_{L}(r)\|=5$. This proves that the constant $5$ is sharp.
\end{proof}
\par\medskip
Note that the last theorem is actually a special case of a more
general result. The properties we have used are that the family of
convex harmonic mappings is affine and linear invariant and that the
value of the supremum of $|h''(0)|$ among all functions $f\in C_H^0$
is known. Therefore, the theorem still holds for any subclass of
$S_H$ that is invariant under normalized affine transformations and
disk automorphisms.
%%%%%%%%%%%%%%%%%%%%%%%%%%%%%%%%%%%%%%%%%%%%%%%%%%%%%%%%%%%%%%%%%%%%%%%%%%%%%%%%%%%%%%%%%%%%%%%
\subsection{The Krauss theorem for harmonic mappings}\label{ssec-KraussHarmonic}
%%%%%%%%%%%%%%%%%%%%%%%%%%%%%%%%%%%%%%%%%%%%%%%%%%%%%%%%%%%%%%%%%%%%%%%%%%%%%%%%%%%%%%%%%%%%%%%

We prove that an analogue of the well-known theorem by Krauss holds for univalent harmonic mappings in the unit disk. To prove our result, we need to use Theorem~9 in \cite{SS} that states what follows: there exists a positive real constant $R$ such that if $f$ is a univalent harmonic mapping in $\D$ normalized by $h(0)=g(0)=1-h^\prime(0)=0$ (in other words, if $f\in S_H$), then $f_R(z)=f(Rz)$ is convex (and univalent) in the unit disk.

\begin{thm}\label{thm-kraussharmonic} There exists a positive constant $C$ such that
$\|S_f\|\leq C$ for all univalent harmonic mappings $f$ in the unit
disk.
\end{thm}

\begin{proof} By Proposition~\ref{prop-afinne}, we may assume that $f\in S_H$. Notice that if $f$ is univalent and $\phi$ is an
automorphism of the disk with $\phi(0)=z$, then $f\circ \phi$ is
also univalent and $|S_{(f\circ\phi)}(0)|=|S_f(z)|(1-|z|^2)^2$.
Hence, it is not difficult to prove that
\[
\sup_{f\in S_H} \|S_f\|=\sup_{f\in S_H} |S_f(0)|\,.
\]
\par
Now, fix an arbitrary function $f\in S_H$. According to \cite[Theorem 9]{SS}, $f_R(z)=f(Rz)$ is a univalent convex harmonic mapping. Therefore, $|S_{f_R}(0)|\leq 6$ by Proposition~\ref{prop-convex}. Using the chain rule, we see that
\[
R^2 |S_f(0)|=|S_{f_R}(0)|\leq 6\,.
\]
Therefore, $\sup_{f\in S_H} |S_f(0)|\leq 6/R^2$. This proves the theorem.
\end{proof}
\par
We end this section with two examples.
\par
\begin{ex} Recall that the analytic Koebe function $k$ has Schwarzian norm $\|Sk\|=6$ and that for each $-1<r<1$, $|Sk(r)|(1-r^2)^2=\|Sk\|$. The Schwarzian derivative of the harmonic Koebe function $K$ defined by (\ref{eq-Koebe-harmonic}) equals
\[
S_K(z)=-\frac{19+10z+3z^2-44|z|^2-10z|z|^2+19|z|^4-10\overline z+10\overline z|z|^2+3\overline z^2}{2(1-z^2)^2(1-|z|^2)^2}\,.
\]
A tedious calculation shows that
$$|S_K(z)|\leq \frac{19}{2}\cdot\frac{1}{(1-|z|^2)^2}\,.$$
Hence, $\|S_K\|=|S_K(0)|=19/2$. Notice that $|S_K(r)|(1-r^2)^2=\|S_K\|$ for all real numbers $r\in(-1,1)$.
\end{ex}
\begin{ex} The Schwarzian derivative of the harmonic function $K_2=h+\overline g$, where $h$ and $g$ are defined by (\ref{eq-formulaK2}), is equal to
$$S_{K_2}(z)=\frac{-2}{(1-z)^2(1-|z|^4)^2}\left(2+|z|^4+\overline z^2-6|z|^2\overline z+2\overline z^2|z|^4\right)\,.$$

It can be checked that
$$\|S_{K_2}\|=\lim_{r\to 1}|S_{K_2}(r)|(1-r^2)^2=19/2=\|S_K\|\,.$$
\end{ex}
%%%%%%%%%%%%%%%%%%%%%%%%%%%%%%%%%%%%%%%%%%%%%%%%%%%%%%%%%%%%%%%%%%%%%%%%%%%%%%%%%%%%%%%%%%%%%%%
\subsection{Harmonic mappings in the unit disk with finite Schwarzian norm}\label{ssec-finiteSchwarziannorm}
%%%%%%%%%%%%%%%%%%%%%%%%%%%%%%%%%%%%%%%%%%%%%%%%%%%%%%%%%%%%%%%%%%%%%%%%%%%%%%%%%%%%%%%%%%%%%%%
The following theorem asserts that a harmonic mapping has finite Schwarzian norm $\|S_f\|$ if and only if its analytic part does. That the same result holds for the Schwarzian norm $\|\mathbb{S}\|$ (restricted to the family $\mathcal{F}$ defined in the beginning of this section) was proved in \cite{Ch-D-O-07}. The proof of our theorem is similar to that in \cite{Ch-D-O-07} but, in this case, we need to prove that there exists a constant $C$ such that for all $z\in\D$
\begin{equation}\label{eq-thm-finiteSch3}
\left|\frac{\omega''(z)\omega(z)}{1-|\omega(z)|^2}\right|(1-|z|^2)^2\leq
C\,.
\end{equation}
The last inequality is a consequence of the results obtained in \cite{MSZ} using boundedness criteria for weighted composition operators between various Bloch-type spaces. We give a direct proof of (\ref{eq-thm-finiteSch3}).

\begin{lem}\label{lema-w''} Let $\omega:\D\to\D$ be an analytic function. Then, there exists a constant
$C$ such that (\ref{eq-thm-finiteSch3} holds).
\end{lem}

\begin{proof} Let $\varphi$ be any convex analytic mapping in the unit disk with $\varphi(0)=0$.
Consider the harmonic mapping $f$ constructed as the horizontal
shear of the function $\varphi$ with dilatation $\omega$. This is,
$f=h+\overline g$, where $h$ and $g$ are obtained by solving the
linear system of equations
\begin{eqnarray}
\nonumber \quad \left\{ \begin{array}{c} h(z)-g(z)=\varphi(z)  \\ \nonumber
g^\prime(z)/h^\prime(z)=\omega(z)  \end{array}\right. \quad
(z\in\D)\,,
\end{eqnarray}
with $h(0)=g(0)=0$.  By Theorem 5.17 in \cite{CSS}, since $\varphi$
is convex, the harmonic mapping $f$ is univalent and $h$ is
close-to-convex. Therefore, using Theorem~\ref{thm-kraussharmonic}
and the well-known results for univalent analytic functions, we have
that  $\|S_f\|$, $\|Sh\|$, and $|\left(h''/h'\right)(z)|(1-|z|^2)$
are bounded in the unit disk. By the Schwarz-Pick lemma, we obtain
that
$$\frac{|\omega(z)\omega'(z)|(1-|z|^2)}{1-|\omega(z)|^2}$$ is bounded by $1$ in $\D$.
\par
Using the definition of the Schwarzian derivative (\ref{eq-newSchwarzian}) and the triangle inequality, we see that
$$\frac{\left|\omega(z)\omega''(z)\right| (1-|z|^2)^2}{1-|\omega(z)|^2}\leq \|Sh\|+\|S_f\|+\left|\frac{h''}{h'}(z)\right|(1-|z|^2)+\frac 32\leq C<\infty\,.$$
\end{proof}
Now we have all the tools to prove the next theorem.
\begin{thm}\label{thm-finiteSchw} Let $f=h+\overline g$ be a locally univalent harmonic mapping in the unit disk. Then $\|S_f\|<\infty$ if and only if
$\|Sh\|<\infty$.
\end{thm}

\begin{proof} As we mentioned before, since $S_{\overline f}=S_f$, there is no loss of generality if we assume that $f$ is sense-preserving. Suppose that $\|Sh\|<\infty$. By the triangle inequality, we see that
\[
|S_f|\leq |Sh|+\frac{|ww^\prime|}{1-|w|^2}\, \left|\frac{h^{\prime\prime}}{h^\prime}\right|+\frac{|ww^{\prime\prime}|}{1-|w|^2}+\frac 32 \left|\frac{w w^\prime}{1-|w|^2}\right|^2\,.
\]

By hypotheses, there exists a constant $C_1$ such that
\[
|Sh(z)|(1-|z|^2)^2\leq C_1\,,\quad z\in\D\,.
\]

Since $f$ is sense-preserving, its dilatation $\omega$ maps the unit disk into itself. Using the
Schwarz-Pick Lemma, we have that
\[
\frac{\left|\omega'(z)\omega(z)\right|(1-|z|^2)}{1-|\omega(z)|^2}\leq 1\,.
\]

By Lemma~\ref{lema-w''}, there exists a constant $C_2$ such that
$$\frac{\left|\omega(z)\omega''(z)\right|(1-|z|^2)^2}{1-|\omega(z)|^2}\leq C_2\,.$$

Finally, a result of Pommerenke (see \cite{Pom1}) asserts that
\begin{equation}\label{eq-thm-finiteSch1}
(1-|z|^2)\left|\frac{h''}{h'}(z)\right|\leq 2+2\sqrt{1+\frac
12\|Sh\|}\,.
\end{equation}
Putting all the estimates together, we obtain
\[
\|S_f\|\leq C_1+2\sqrt{1+\frac
12C_1}+C_2+\frac 72<\infty\,.
\]
Conversely, suppose that $\|S_f\|<\infty$. Using formula (\ref{eq-newSchwarzian}) and the preceding estimates we have
\begin{equation}\label{eq-thm-finiteSch2}
|Sh(z)|\leq |S_f(z)|+\frac{1}{1-|z|^2}\left|\frac{h^{\prime\prime}}{h^\prime}(z)\right|+\frac{C}{(1-|z|^2)^2}\,.
\end{equation}
In order to use (\ref{eq-thm-finiteSch1}), we apply inequality (\ref{eq-thm-finiteSch2}) to the dilation
 $f_r(z)=f(rz)=h_r+\overline{g_r}$, where $0<r<1$. By the chain rule, we have that $S_{f_r}(z)=r^2S_f(rz)$. Therefore,
$$|S_{f_r}(z)|(1-|z|^2)^2\leq |S_{f}(rz)|(1-|rz|^2)^2 \leq \|S_f\|\,.$$
Since $\|Sh_r\|$ is finite for each $r<1$, we can use (\ref{eq-thm-finiteSch1}) to get
$$\|Sh_r\|-2\sqrt{1+\frac
12\|Sh_r\|}\leq \|S_f\|+C+2\,.$$
Let $r\to 1$ to obtain $\|Sh\|<\infty$.
\end{proof}
\par
To conclude this section, we would like to mention explicitly
Theorem 3 from \cite{Ch-D-O-08}. This theorem states that given any
sense-preserving harmonic mapping $f=h+\overline g$ whose dilatation
$\omega=g^\prime/h^\prime$ is the square of an analytic function in
the unit disk, then $\|\mathbb{S}f\|<\infty$ if and only if $f$ is
uniformly locally univalent. The proof of this result uses that
$\|\mathbb{S}f\|<\infty$ if and only if $\|Sh\|<\infty$. The
requirement on the dilatation is only needed to define the
Schwarzian derivative $\mathbb{S}f$. By
Theorem~\ref{thm-finiteSchw}, we know that $\|Sh\|<\infty$ if and
only if $\|S_f\|<\infty$. Using this fact and the argument used in
the proof of \cite[Theorem 3]{Ch-D-O-08}, we obtain the next result
(which does not require any extra condition on the dilatation).
\begin{thm}
Let $f=h+\overline g$ be a sense-preserving harmonic mapping in the
unit disk. Then, $\|S_f\|<\infty$ if and only if $f$ is uniformly
locally univalent.
\end{thm}
% xxxxxxxxxxxxxxxxxxxxxxxxxxxxxxxxxxxxxxxxxxxxxxxxxxxxxxxxxxxxxxxxxxxxxxxxxxxxxxxxxxxxxxxxxxxxxxxxxxxxxxxxxxxxxxxxxxxxxxxxxxxxxxxx
\section{A Becker-type criterion of univalence}\label{sec-univalence}
% xxxxxxxxxxxxxxxxxxxxxxxxxxxxxxxxxxxxxxxxxxxxxxxxxxxxxxxxxxxxxxxxxxxxxxxxxxxxxxxxxxxxxxxxxxxxxxxxxxxxxxxxxxxxxxxxxxxxxxxxxxxxxxxx
We finish this paper with a criterion of univalence for
sense-preserving harmonic mappings.
\begin{thm}\label{thm-Becker1}
Let $f=h+\overline g$ be a sense-preserving harmonic function in the
unit disk with dilatation $\omega$. If for all $z\in \D$
\begin{equation}\label{eq-Becker1}
|z P_f(z)|+\frac{|z\omega^\prime(z)|}{1-|\omega(z)|^2}\leq
\frac{1}{1-|z|^2}\,,
\end{equation}
then $f$ is univalent. The constant $1$ is sharp.
\end{thm}
Before proving Theorem~\ref{thm-Becker1}, we would like to remark
that when $\omega$ is equal to a constant $\alpha\in \D$, say, then
(\ref{eq-Becker1}) becomes
\[
\left|z\frac{h''}{h^\prime}(z)\right|\leq \frac{1}{1-|z|^2}\,,
\]
which implies that $h$ is univalent. Hence, $f$ (that is necessarily
of the form $h+\alpha \overline h$) is also univalent. In other
words, Theorem~\ref{thm-Becker1} generalizes the classical Becker's
criterion of univalence for analytic functions.
\par\medskip
\begin{pf}
Using (\ref{eq-preschw-jacobian}), the triangle inequality, that
$\omega$ maps the unit disk into itself, and assuming that
(\ref{eq-Becker1}) holds, we obtain
\begin{eqnarray}
\nonumber \left|z\frac{h''}{h^\prime}(z)\right|&\leq& |z
P_f(z)|+\frac{|z\overline{\omega(z)}\omega^\prime(z)|}{1-|\omega(z)|^2}\\ \nonumber
& \leq& |z P_f(z)|+\frac{|z\omega^\prime(z)|}{1-|\omega(z)|^2} \leq
\frac{1}{1-|z|^2}\,.
\end{eqnarray}
Therefore, $h$ is univalent by Becker's criterion.
\par
For an arbitrary $a\in\D$, consider the function $f_a=f-\overline{a
f}=h_a+\overline{g_a}$\,. The dilatation $\omega_a$ of $f_a$ equals
$\phi_a\circ \omega$ where $\phi_a$ is the automorphism of the disk
defined by (\ref{eq-automorphism}). Then, $P_{f_a}=P_f$ (by Proposition~\ref{prop-afinne}) and
\[
\frac{|z\omega_a^\prime(z)|}{1-|\omega_a(z)|^2}=\frac{|z\omega^\prime(z)|}{1-|\omega(z)|^2}\,.
\]
Thus, (\ref{eq-Becker1}) holds for $f_a$ as well. We conclude that
the analytic parts $h_a=h+ag$ of $f_a$ are univalent for all $a \in \D$. Using
Hurwitz's theorem (and that $f$ is sense-preserving), we get that
$h+\lambda g$ are univalent for all $|\lambda|=1$. This implies that
$f_\lambda=h+\lambda\overline g$ are univalent for all such
$\lambda$ (see \cite{HM}). In particular, we obtain that $f$ is
univalent.
\par
Since the constant $1$ is sharp for analytic functions, it follows
that this constant is sharp for the harmonic mappings too.
\end{pf}

\section*{Acknowledgements}
We would like to thank Professors Martin Chuaqui and Dragan
Vukoti\'c for their helpful comments that improved the clarity of
the exposition in this paper.

% xxxxxxxxxxxxxxxxxxxxxxxxxxxxxxxxxxxxxxxxxxxxxxxxxxxxxxxxxxxxxxxxxxxxxxxxxxxxxxxxxxxxxxxxxxxxxxxxxxxxxxxxxxxxxxxxxxxxxxxxxxxxxxxxxxxxx


\begin{thebibliography}{10}

\bibitem{A} L. Ahlfors, Sufficient conditions for quasiconformal
extension, \emph{Ann. Math. Studies} \textbf{79} (1974), 23--29.

\bibitem{AW} L. Ahlfors and G. Weill, A uniqueness theorem for Beltrami equations, \emph{Proc. Amer. Math. Soc.} \textbf{13} (1962), 975--978.

\bibitem{B} J. Becker, L\"{o}wnersche differentialgleichung und quasikonform fortsetzbare schlichte functionen, \emph{J. Reine Angew. Math.}
\textbf{255} (1972) 23--43.

\bibitem{BP} J. Becker and Ch. Pommerenke, Schlichtheitskriterien und Jordangebiete, \emph{J. Reine Angew. Math.}
\textbf{354} (1984) 74--94.

\bibitem{CHDO-DefSchw}
M. Chuaqui, P. Duren, and B. Osgood, The Schwarzian derivative for
harmonic mappins, \emph{J. Anal. Math.} \textbf{91} (2003), 329--351.

\bibitem{Ch-D-O-04}
M. Chuaqui, P. Duren, and B. Osgood, Curvature properties of planar harmonic mappings,
\emph{Comput. Methods Funct. Theory} \textbf{4} (2004), 127--142.

\bibitem{Ch-D-O-05}
M. Chuaqui, P. Duren, and B. Osgood, Ellipses, near ellipses, and harmonic M\"{o}bius transformations,
\emph{Proc. Amer. Math. Soc.} \textbf{133} (2005), 2705--2710.

\bibitem{Ch-D-O-07}
M. Chuaqui, P. Duren, and B. Osgood, Univalence criteria for lift
harmonic mappings to minimal surface, \emph{J. Geom. Anal.} \textbf{17} (1) (2007),
49--74.

\bibitem{Ch-D-O-07Phill}
M. Chuaqui, P. Duren, and B. Osgood, Schwarzian derivative criteria for
valence of analytic and harmonic mappings, \emph{Math. Proc. Cambridge
Phil. Soc.} \textbf{143} (2007), 473--486.

\bibitem{Ch-D-O-08}
M. Chuaqui, P. Duren, and B. Osgood, Schwarzian derivative and uniform
local univalence, \emph{Comput. Methods Funct. Theory} \textbf{8} (2008), 21--34.

\bibitem{CSS}%2
J. Clunie and T. Sheil-Small, Harmonic univalent functions,
\emph{Ann. Acad. Sci. Fenn. Ser. A. I} \textbf{9} (1984), 3--25.

\bibitem{Dur-Univ}
P. Duren, \emph{Univalent Functions}, Springer-Verlag, New York, 1983.

\bibitem{Dur-Harm}
P. Duren, \emph{Harmonic Mappings in the Plane}, Cambridge University
Press, Cambridge, 2004.

\bibitem{DL}
P. Duren and O. Lehto, Schwarzian derivatives and homeomorphic
extensions, \emph{Ann. Acad. Sci. Fenn, Ser A I}, no. 477 (1970), 11
pp.

\bibitem{GK}
I. Graham and G. Kohr, \emph{Geometric function theory in one and
higher dimensions} Marcel Dekker Inc., New York, Basel, 2003.

\bibitem{HM}
R. Hern\'andez and M. J. Mart\'{\i}n, Stable geometric properties of
analytic and harmonic functions. Preprint. Available from
http://www.uam.es/mariaj.martin.
\bibitem{Krauss}%5
W. Krauss, \"{U}ber den Zusammenhang einiger Charakteristiken eines einfach zusammenh\"{a}ngenden Bereiches mit der Kreisabbildung, \emph{Mitt. Math. Sem. Giessen} \textbf{21} (1932), 1--28.

\bibitem{Lewy}
H. Lewy, On the non-vanishing of the jacobian in certain
one-to-one mappings, \emph{Bull. Amer. Math. Soc.} \textbf{42} (1936), 689--692.

\bibitem{MSZ}
B. D. MacCluer, K. Stroethoff, and R. Zhao, Generalized Schwarz-Pick estimates, \emph{Proc. Amer. Math. Soc.} \textbf{131} (2002), 593--599.

\bibitem{Nehari}
Z. Nehari, The Schwarzian derivative and schlicht functions, \emph{Bull. Amer. Math. Soc.} \textbf{55} (1949), 545--551.

\bibitem{OS}
B. Osgood and D. Stowe, The Schwarzian derivative and conformal
mappings of Riemannian manifolds, \emph{Duke Math. J.} \textbf{67} (1992), 57--99.

\bibitem{Pom1}
Ch. Pommerenke, Linear-invariante Familien analytischer Funktionen {I},
\emph{Math. Ann.} \textbf{155} (1964), 108--154.

\bibitem{P}%6
Ch. Pommerenke, \emph{Univalent Functions}, Vandenhoeck $\&$
Ruprecht, G\"{o}ttingen, 1975.

\bibitem{SS}
T. Sheil-Small, Constants for planar harmonic mappings, \emph{J. London Math. Soc.} \textbf{42} (1990), 237--248.

\bibitem{T}
H. Tamanoi, Higher Schwarzian operators and combinatorics of the
Schwarzian derivative, \emph{Math. Ann.} \textbf{305} (1996), 127--151.

\bibitem{Y}
S. Yamashita, Norm estimates for function starlike or convex of
order alpha, \emph{Hokkaido Math. J.} \textbf{28} (1999), 217-230.

\end{thebibliography}
\end{document}